\documentclass[10pt, a4paper]{article}
\usepackage{amsfonts}
\usepackage{mathrsfs}
\usepackage{latexsym}
\usepackage{xy}
\usepackage{amsfonts,amsmath,amssymb,amsthm}
\usepackage{color}
\xyoption{all}

\newcommand{\bcen}{\begin{center}}     \newcommand{\ecen}{\end{center}}
\newcommand{\bay}{\begin{array}}      \newcommand{\eay}{\end{array}}
\newcommand{\beq}{\begin{eqnarray*}}      \newcommand{\eeq}{\end{eqnarray*}}
\def\az{\alpha}
\def\bz{\beta}
\def\dz{\delta}

\def\ot{\otimes}

\def\wt{\widetilde}
\def\mt{\mapsto}

\def\dz{\delta}

\def\Ext{\mathrm{Ext}}

\def\Hom{\mathrm{Hom}}

\def\id{\mathrm{id}}
\def\Im{\mathrm{Im}}

\def\Ker{\mathrm{Ker}}

\def\mt{\mapsto}

\def\ol{\overline}
\def\op{\mathrm{op}}
\def\ot{\otimes}

\def\per{\mathrm{per}}

\def\RHom{\mathrm{RHom}}

\def\T{\mathbb{T}}

\begin{document}

\newtheorem{theorem}{Theorem}
\newtheorem{proposition}{Proposition}
\newtheorem{lemma}{Lemma}
\newtheorem{corollary}{Corollary}
\newtheorem{remark}{Remark}
\newtheorem{example}{Example}
\newtheorem{definition}{Definition}
\newtheorem*{conjecture}{Conjecture}
\newtheorem{question}{Question}

\title{Exact Hochschild extensions and deformed Calabi-Yau completions}

\author{Yang Han, Xin Liu and Kai Wang}

\date{\footnotesize KLMM, Academy of Mathematics and Systems Science,
Chinese Academy of Sciences, \\ Beijing 100190, China. \\ School of Mathematical Sciences, University of
Chinese Academy of Sciences, \\ Beijing 100049, China.\\ E-mail:
hany@iss.ac.cn (Y. Han), liuxin215@mails.ucas.ac.cn (X. Liu),\\ wangkai@amss.ac.cn (K. Wang)}

\maketitle

\begin{abstract} We introduce the Hochschild extensions of dg algebras, which are $A_\infty$-algebras.
We show that all exact Hochschild extensions are symmetric Hochschild extensions, more precisely,
every exact Hochschild extension of a finite dimensional complete typical dg algebra is a symmetric $A_\infty$-algebra.
Moreover, we prove that the Koszul dual of trivial extension is Calabi-Yau completion
and the Koszul dual of exact Hochschild extension is deformed Calabi-Yau completion,
more precisely, the Koszul dual of the trivial extension of a finite dimensional complete dg algebra is the Calabi-Yau completion of its Koszul dual,
and the Koszul dual of an exact Hochschild extension of a finite dimensional complete typical dg algebra is the deformed Calabi-Yau completion of its Koszul dual.
\end{abstract}

\medskip


\medskip


\tableofcontents

\section{Introduction}

Throughout this paper, $k$ is a field and $K=k^t$ for a positive integer $t$.

Symmetric algebras introduced by Brauer and Nesbitt \cite{BraNes37} is an important class of finite dimensional algebras
(see \cite{SkoYam11,Yam96} and the references therein).
For every finite dimensional algebra $A$, one can construct its trivial extension $\T(A):= A\ltimes A^\vee$ with $A^\vee=\Hom_k(A,k)$,
which is always symmetric \cite{AusReiSma95}. It means that symmetric algebras are as many as finite dimensional algebras.
For each Hochschild 2-cocycle $\az$ of $A$ with coefficients in $A^\vee$, one can construct its Hochschild extension $\T(A,\az)$.
In general, $\T(A,\az)$ is possibly not a symmetric algebra.
In 1999, Ohnuki, Takeda and Yamagata gave a sufficient condition for $\T(A,\az)$ to be symmetric \cite{OhnTakYam99}.
Recently, Itagaki provided a weaker sufficient condition \cite{Ita19}.

Calabi-Yau algebras were introduced by Ginzburg \cite{Gin06}.
Their bounded derived categories are Calabi-Yau triangulated categories \cite{Kel08}.
There exist Van den Bergh dualities between their Hochschild homologies and Hochschild cohomologies \cite{Van98}.
Furthermore, their Hochschild cohomologies are Batalin-Vilkovisky algebras \cite{Gin06,Abb15}.
For a homologically smooth dg algebra $A$, Keller introduced its Calabi-Yau completion $\Pi_n(A)$ which is an exact Calabi-Yau dg algebra \cite{Kel11,Kel18}.
More general, for a Hochschild class $[\az]\in HH_{n-2}(A)$,
he introduced its deformed Calabi-Yau completion $\Pi_n(A,\az)$ called derived preprojective algebra as well \cite{Kel11}.
If $[\az] \in HH_{n-2}(A)$ is an almost exact Hochschild homology class, i.e., it is the image of a negative cyclic homology class,
then $\Pi_n(A,\az)$ is an almost exact Calabi-Yau algebra \cite{Yeu16}.
Therefore, Ginzburg dg algebras associated to quivers with potential are Calabi-Yau dg algebras.

Koszul algebras were introduced by Priddy \cite{Pri70}.
They admit Koszul resolutions which are simpler than bar resolutions.
Koszul duality was introduced by Beilinson, Ginzburg and Schechtman \cite{BeiGinSch88},
and further developed by Beilinson, Ginzburg and Soergel \cite{BeiGinSoe96}.
There exist Koszul dualities on three levels --- algebras, module categories and derived categories.
Heretofore, Koszul duality has been built for dg categories \cite{Kel94}, operads \cite{GinKap94},
props \cite{Val07}, $A_\infty$-algebras \cite{LuPalWuZha08},
curved props \cite{HirMil12}.
For an augmented dg algebra $A$, there are two ways to define its Koszul dual.
One is $(BA)^\vee$, the graded dual of the bar construction $BA$ of $A$.
The other is $A^\dag:=\Omega(A^\vee)$, the cobar construction of the graded dual $A^\vee$ of $A$,
in the case that $A$ is a locally finite, bounded above or bounded below, augmented dg algebra.
In practice, $\Omega(A^\vee)$ is more feasible than $(BA)^\vee$.
Once $A$ is a typical dg algebra, then the Koszul dual $A^\dag=\Omega(A^\vee) \cong (BA)^\vee$ (see \cite{HanLiuWan18}).
Koszul dual $(-)^\dag$ sets up the relationship between finite dimensional dg algebras and homologically smooth dg algebras \cite{Lun10,HanLiuWan18},
and the relationship between symmetric dg algebras and Calabi-Yau dg algebras \cite{Van15,Her18,HanLiuWan18}.
Moreover, the Hochschild (co)homologies of dg algebras and their Koszul duals are closely related \cite{Her15,HanLiuWan18}.

The higher trivial extensions of algebras have already been studied
in the works of Keller \cite{Kel05}, Amiot \cite{Ami09} and L.Y. Guo \cite{Guo11} during introducing generalized cluster categories.
Recently, Guo, Grant and Iyama found some connections between higher trivial extensions and higher preprojective algebras of algebras
under Koszul duality \cite{Guo19,GraIya19}. Nonetheless, many trivial extensions, more general, Hochschild extensions of dg algebras, have not been studied yet.

In this paper, firstly, we will introduce the {\it Hochschild extension} $\T(A,M,\az)$ of an augmented dg $K$-algebra $A$
by a dg $A$-bimodule $M$ satisfying $M_{-2}=0$ and a Hochschild 2-cocycle $\az : A^{\ot 2}\to M$.
It is an augmented $A_\infty$-algebra (Theorem \ref{Theorem-Hoch-Ext}).
Secondly, we focus on the Hochschild extension $\T_n(A,\az):=\T(A,A^\vee[-n],\az)$ of a finite dimensional complete typical dg $K$-algebra $A$
by the shift $A^\vee[-n]$ of the graded dual dg $A$-bimodule $A^\vee$ of $A$ satisfying $A_{2-n}=0$ and a Hochschild 2-cocycle $\az : A^{\ot 2}\to A^\vee[-n]$.
An augmented $A_\infty$-algebra $\T$ is said to be {\it $n$-symmetric} if $\T \cong \T^\vee[-n]$ as $A_\infty$-$\T$-bimodules.
A Hochschild extension $\T_n(A,\az)$ is said to be {\it symmetric} if $\T_n(A,\az)$ is a symmetric $A_\infty$-algebra.
In general, a Hochschild extension $\T_n(A,\az)$ might not be symmetric.
We will give a cohomological criterion for a Hochschild extension $\T_n(A,\az)$ to be symmetric.
A Hochschild extension $\T_n(A,\az)$ is said to be {\it exact} if the second Hochschild cohomology class
$[\az]\in H^2(A,A^\vee[-n])=HH_{2-n}(A)^\vee$ is {\it exact}, i.e., in the image of the natural map $I_{2-n}^\vee:HC_{2-n}(A)^\vee \to HH_{2-n}(A)^\vee$.
We will show that exact Hochschild extensions of finite dimensional complete typical dg $K$-algebras are symmetric (Theorem \ref{Theorem-HochExt-Sym}).
It is a cohomological generalization of Ohnuki, Takeda and Yamagata's result \cite[Theorem 2.2]{OhnTakYam99}.
Thirdly, utilizing the relation between the Hochschild homologies of a complete typical dg $K$-algebra $A$ and its Koszul dual $A^\dag$, i.e.,
the isomorphism $HH_n(A)^\vee \cong HH_{-n}(A^\dag)$ (Theorem \ref{Theorem-HHHH_Connes}),
we define the {\it Koszul dual} of a Hochschild cohomology class $[\bz]\in H^n(A,A^\vee)=HH_n(A)^\vee$
to be its image $[\bz^\dag]\in HH_{-n}(A^\dag)$ under the isomorphism $HH_n(A)^\vee \cong HH_{-n}(A^\dag)$.
Employing the relations between the Hochschild (co)homologies and (negative) cyclic homologies of
a complete typical dg $K$-algebra and its Koszul dual, i.e.,
the isomorphisms $HH_n(A)^\vee \cong HH_{-n}(A^\dag)$ and $HC_n(A)^\vee \cong HN_{-n}(A^\dag)$ (Proposition \ref{Proposition-HC-KoszulDual}),
we show that a Hochschild cohomology class $[\bz]\in H^n(A,A^\vee)=HH_n(A)^\vee$ is exact if and only if its Koszul dual
$[\bz^\dag]\in HH_{-n}(A^\dag)$ is {\it almost exact}, i.e., in the image of the natural map $HN_{-n}(A^\dag) \to HH_{-n}(A^\dag)$ (Proposition \ref{Proposition-Ex-AlmostEx}).
Fourthly, the {\it $n$-trivial extension} $\T_n(A)$ of a finite dimensional complete dg $K$-algebra $A$, i.e.,
the augmented dg $K$-algebra $A\oplus A^\vee[-n]$ with the product given by $(a,s^{-n}f)\cdot(a',s^{-n}f'):=(aa',s^{-n}((-1)^{|a|n}af'+fa'))$, is obviously an $n$-symmetric dg $K$-algebra.
Since the Koszul dual $A^\dag$ of $A$ is a homologically smooth dg $K$-algebra \cite[Theorem 7]{HanLiuWan18},
we can construct its $n$-Calabi-Yau completion $\Pi_n(A^\dag)$.
We will prove that $\Pi_n(A^\dag) \cong \T_n(A)^\dag$ and they are both exact $n$-Calabi-Yau dg algebras (Theorem \ref{Theorem-TrivExt-CYComp}).
This isomorphism implies that the Koszul dual of trivial extension is Calabi-Yau completion,
which could be illustrated as the following:
$$\xymatrixcolsep{7pc}\xymatrixrowsep{3pc}\xymatrix{ A \ar@{.>}[rr]^-{\mbox{trivial extension}}
\ar@{.>}[d]^-{\mbox{Koszul dual}} && \T_n(A) \ar@{.>}[d]_-{\mbox{Koszul dual}} \\
A^\dag \ar@{.>}[r]_-{\mbox{CY completion}} & \Pi_n(A^\dag) \ar[r]^-{\cong}_-{\mbox{Theorem A}} & \T_n(A)^\dag .}$$
It will be applied to recover a result of J.Y. Guo in \cite{Guo19},
and could be viewed as the dg lift of the isomorphism in \cite[Theorem 5.3]{Guo19}.
Last but not least, for a finite dimensional complete typical dg $K$-algebra $A$ and an exact Hochschild cohomology class $[\az]\in H^2(A,A^\vee[-n])=HH_{2-n}(A)^\vee$,
on one hand, we can construct the Koszul dual $\T_n(A,\az)^\dag$ of the exact Hochschild extension $\T_n(A,\az)$ of $A$ by $A^\vee[-n]$ and $\az$.
On the other hand, since the Koszul dual $A^\dag$ of $A$ is homologically smooth,
we can construct the deformed Calabi-Yau completion $\Pi_n(A^\dag,\az^\dag)$
of $A^\dag$ by the almost exact Hochschild homology class $[\az^\dag]\in HH_{n-2}(A^\dag)$, i.e., the Koszul dual of $[\az]$.
We will show that $\T_n(A,\az)^\dag \cong \Pi_n(A^\dag,\az^\dag)$ and they are both almost exact $n$-Calabi-Yau dg algebras (Theorem \ref{Theorem-ExHochExt-DefCYComp}).
This isomorphism implies that the Koszul dual of exact Hochschild extension is deformed Calabi-Yau completion,
which could be illustrated as the following:
$$\xymatrixcolsep{9pc}\xymatrixrowsep{3pc}\xymatrix{ A \ar@{.>}[rr]^-{\mbox{exact Hochschild extension}}
\ar@{.>}[d]^-{\mbox{Koszul dual}} && \T_n(A,\alpha) \ar@{.>}[d]_-{\mbox{Koszul dual}} \\
A^\dag \ar@{.>}[r]_-{\mbox{deformed CY completion}} & \Pi_n(A^\dag,\alpha^\dag) \ar[r]^-{\cong}_-{\mbox{Theorem B}} & \T_n(A,\alpha)^\dag .}$$

Throughout, $\ot:=\ot_K$ and $(-)^\vee=\Hom_k(-,k)$ the graded $k$-dual.
Denote by $\mathbb{N}$ the set of positive integers, and by $\mathbb{N}_0$ the set of nonnegative integers.
By convention, in a complex, an element of lower degree $i \in \mathbb{Z}$ is of upper degree $-i$.

\section{Preliminaries}

In this section, we fix some terminologies and notations on (co)coaugmented dg $K$-(co)algebras and (co)augmented $A_\infty$-(co)algebras,
some of which are quite not consistent in existing literatures.

\subsection{(Co)Augmented dg $K$-(co)algebras}

We mainly refer to \cite{Lef03,LuPalWuZha04,HanLiuWan18} for some knowledge on dg $K$-(co)algebras.

\bigskip

\noindent{\bf Augmented dg $K$-algebras.}
An {\it augmented dg $K$-algebra} $A = (A,d,\mu,\eta,\varepsilon)$ is a $K$-bimodule complex $(A,d)$ equipped with
three $K$-bimodule complex morphisms $\mu: A \ot A \rightarrow A$ called {\it product},
$\eta : K \rightarrow A$ called {\it unit},
and $\varepsilon : A \rightarrow K$ called {\it augmentation},
satisfying associativity $\mu \circ (\mu \ot \id_A) = \mu \circ (\id_A \ot \mu)$,
unitality $\mu \circ (\eta \ot \id_A) = \id_A = \mu \circ (\id_A \ot \eta)$,
and $\varepsilon \circ \eta = \id_K$.
If $A$ is an augmented dg $K$-algebra then $A = K 1_A \oplus \ol{A}$ where $1_A=\eta(1_K)$ and $\ol{A} = \Ker \varepsilon$
called the {\it augmentation ideal} of $A$. We always identify $K 1_A$ with $K$.

A {\it morphism} from an augmented dg $K$-algebra $(A,d_A,\mu_A,\eta_A,\varepsilon_A)$ to an augmented dg $K$-algebra $(A',d_{A'},\mu_{A'},\eta_{A'},\varepsilon_{A'})$
is a $K$-bimodule complex morphism $f : A \rightarrow A'$
satisfying $f \circ \mu_A = \mu_{A'} \circ (f \ot f)$, $f \circ \eta_A = \eta_{A'}$ and $\varepsilon_{A'} \circ f = \varepsilon_A$.

An augmented dg $K$-algebra $A$ is {\it complete} if $\bigcap\limits_{n\in\mathbb{N}} \ol{A}^n=0$.
Obviously, a finite dimensional augmented dg $K$-algebra $A$ is complete if and only if its augmentation ideal $\ol{A}$ is nilpotent,
i.e., there is $n\in\mathbb{N}$ such that $\ol{A}^n=0$.

\bigskip

\noindent{\bf Coaugmented dg $K$-coalgebras.}
A {\it coaugmented dg $K$-coalgebra} $C = (C,d,\linebreak \Delta,\varepsilon,\eta)$ is a $K$-bimodule complex $(C,d)$
equipped with three $K$-bimodule complex morphisms $\Delta: C \rightarrow C \ot C$ called {\it coproduct},
$\varepsilon : C \rightarrow K$ called {\it counit}, and $\eta : K \rightarrow C$ called {\it coaugmentation},
satisfying coassociativity $(\Delta \ot \id_C) \circ \Delta = (\id_C \ot \Delta) \circ \Delta$,
counitality $(\varepsilon \ot \id_C) \circ \Delta = \id_C = (\id_C \ot \varepsilon) \circ \Delta$,
and $\varepsilon \circ \eta = \id_K$.
If $C$ is a coaugmented dg $K$-coalgebra then $C = K1_C \oplus \ol{C}$ where $1_C = \eta(1_K)$ and
$\ol{C} = \Ker \varepsilon$ called the {\it coaugmentation coideal} of $C$. We always identify $K1_C$ with $K$.

A {\it morphism} from a coaugmented dg $K$-coalgebra $(C,d_C,\Delta_C,\varepsilon_C,\eta_C)$ to a coaugmented dg $K$-coalgebra
$(C',d_{C'},\Delta_{C'},\varepsilon_{C'},\eta_{C'})$
is a $K$-bimodule complex morphism $f : C \rightarrow C'$
satisfying $\Delta_{C'} \circ f = (f \ot f) \circ \Delta_C$,
$\varepsilon_{C'} \circ f = \varepsilon_C$ and $f \circ \eta_C = \eta_{C'}$.

Let $C$ be a coaugmented dg $K$-coalgebra.
Define $\ol{\Delta} : \ol{C} \rightarrow \ol{C} \ot \ol C$ by
$\ol{\Delta}(c) = \Delta(c) - 1 \ot c - c \ot 1$ for all $c\in \ol{C}$, and further
$\ol{\Delta}^{(n)} : \ol{C} \rightarrow \ol{C}^{\ot n+1}$
by $\ol{\Delta}^{(0)}=\id_{\ol{C}}, \ \ol{\Delta}^{(1)}=\ol{\Delta},$ and $\ol{\Delta}^{(n)}
= (\ol{\Delta} \ot \id_{\ol{C}}^{\ot n-1}) \circ \ol{\Delta}^{(n-1)}$ for all $n \geq 2$.
Let $F_0C := K$ and $F_nC := K \oplus \Ker \ol{\Delta}^{(n)}$ for $n \geq 1$.
Then $F_nC$ is a dg $C$-bicomodule for all $n \geq 0$. The series
$F_0C \subseteq \cdots \subseteq F_nC \subseteq \cdots$ is called the {\it coradical series} of $C$.
A coaugmented dg $K$-coalgebra $C$ is {\it cocomplete} or {\it conilpotent} if $C = \bigcup\limits_{n\in\mathbb{N}_0} F_nC$.
Clearly, a locally finite, bounded above or below, augmented dg $K$-algebra $A$ is complete
if and only if its graded dual $A^\vee$ is a cocomplete coaugmented dg $K$-coalgebra.

\bigskip

\noindent{\bf Bar construction.} Let $A = K \oplus \ol{A}$ be an augmented dg $K$-algebra and
$$T(s\ol{A}) = \bigoplus\limits_{n\in\mathbb{N}_0} (s\ol{A})^{\ot n} = K \oplus s\ol{A} \oplus (s\ol{A})^{\ot 2} \oplus \cdots$$
the tensor graded $K$-coalgebra, where $s$ is the suspension functor which is also denoted by $[1]$ sometimes.
Write $[a_1|a_2|\cdots|a_n]$ for the homogeneous element $sa_1 \ot sa_2 \ot \cdots \ot sa_n \in (s\ol{A})^{\ot n} \subset T(s\ol{A})$. Let
$$d_0([a_1|\cdots|a_n]) = \sum\limits_{i=1}^{n} (-1)^{|a_1|+\cdots+|a_{i-1}|+i}\ [a_1|\cdots|d_A(a_i)|\cdots|a_n],$$
$$d_1([a_1|\cdots|a_n]) = \sum\limits_{i=1}^{n-1} (-1)^{|a_1|+\cdots+|a_{i}|+i-1}\ [a_1|\cdots|a_ia_{i+1}|\cdots|a_n].$$
Then $d:=d_0+d_1$ is a differential of $T(s\ol{A})$.
The cocomplete coaugmented dg $K$-coalgebra $BA:=(T(s\ol{A}),d)$ is called the {\it bar construction} of $A$.

\bigskip

\noindent{\bf Cobar construction.} Let $C = K \oplus \ol{C}$ be a cocomplete coaugmented dg $K$-coalgebra and
$$T(s^{-1}\ol{C}) = \bigoplus\limits_{n\in\mathbb{N}_0} (s^{-1}\ol{C})^{\ot n} = K \oplus s^{-1}\ol{C} \oplus (s^{-1}\ol{C})^{\ot 2} \oplus \cdots$$
the tensor graded $K$-algebra. Write $\langle c_1|c_2|\cdots|c_n\rangle$ for the homogeneous element
$s^{-1}c_1 \ot s^{-1}c_2 \ot \cdots \ot s^{-1}c_n \in (s^{-1}\ol{C})^{\ot n} \subset T(s^{-1}\ol{C}).$
Let
$$d_0(\langle c_1|\cdots|c_n \rangle) := \sum\limits_{i=1}^{n} (-1)^{|c_1|+\cdots+|c_{i-1}|+i}\ \langle c_1|\cdots|d_C(c_i)|\cdots|c_n \rangle,$$
$$d_1(\langle c_1|\cdots|c_n \rangle) = \sum\limits_{i=1}^{n} (-1)^{|c_1|+\cdots+|c_{i-1}|+|c_{i1}|+i}\ \langle c_1|\cdots|c_{i1}|c_{i2}|\cdots|c_n \rangle,$$
where $\ol{\Delta}(c_i) = c_{i1}\ot c_{i2}$. Note that we always omit $\sum$ and brackets in Sweedler's notation $\Delta(c) = \sum c_{(1)} \ot c_{(2)}$.
Then $d:=d_0+d_1$ is a differential of $T(s^{-1}\ol{C})$.
The augmented dg $K$-algebra $\Omega C:=(T(s^{-1}\ol{C}),d)$ is called the {\it cobar construction} of $C$.

\bigskip

\noindent{\bf Two-sided bar resolutions.} Let $A$ be an augmented dg $K$-algebra.
Define a differential $d$ on the graded $K$-bimodule $A \ot BA \ot A$ by $d := d_{A\ot BA\ot A} + \id_A\ot d^r-d^l\ot\id_A$,
where $d_{A\ot BA\ot A}$ is the differential of the tensor product $A \ot BA \ot A$ of dg $K$-bimodules,
$d^r : BA \ot A \to BA \ot A$ is the composition $(\id_{BA} \ot \mu) \circ (\id_{BA} \ot \pi \ot \id_A) \circ (\Delta \ot \id_A)$,
$d^l : A \ot BA \to A \ot BA$ is the composition $(\mu\ot \id_{BA}) \circ (\id_A \ot \pi \ot \id_{BA}) \circ (\id_A \ot \Delta)$.
Here, $\pi : BA\to A$ is the universal twisting morphism, i.e., the composition $BA \twoheadrightarrow s\ol{A} \stackrel{s^{-1}}{\cong} \ol{A} \hookrightarrow A$.
Then $(A \ot BA \ot A,d)$ is a semi-projective resolution of the dg $A$-bimodule $A$, i.e.,
$(A \ot BA \ot A,d)$ is a semi-projective dg $A$-bimodule and the composition
$$\tilde{\mu} : A \ot BA \ot A \xrightarrow{\id \ot \varepsilon \ot \id} A \ot K \ot A = A \ot A \xrightarrow{\mu} A$$
is a dg $A$-bimodule quasi-isomorphism, called the {\it two-sided bar resolution} of $A$.

\bigskip

\noindent{\bf Typical dg $K$-algebras.} A {\it typical dg $K$-algebra} is a locally finite augmented dg $K$-algebra $A$ which is either non-negative or
non-positive simply connected (i.e., $A_0=K$ and $A_{-1}=0$). The typicality of a dg $K$-algebra $A$
ensures that both the bar construction $BA$ and its graded dual $(BA)^\vee$ are locally finite.
Furthermore, $(BA)^\vee\cong \Omega(A^\vee)$ as dg $K$-algebras.

\subsection{(Co)Augmented $A_\infty$-(co)algebras}

$A_\infty$-algebras (= strongly homotopy associative algebras= sha algebras) were introduced by Stasheff in 1963
as the algebraic counterpart of his theory of $H$-spaces \cite{Sta63}.
We mainly refer to \cite{GetJon90,Kel01,Lef03,LuPalWuZha04,Her18HHA} for some knowledge on $A_\infty$-(co)algebras.

\bigskip

\noindent{\bf Augmented $A_\infty$-algebras.}
An {\it augmented $A_\infty$-algebra} $A=(A,\eta_A,\varepsilon_A,d)$ is a graded $K$-bimodule $A$
together with two graded $K$-bimodule morphisms of degree zero
$\eta_A : K \to A$ called {\it unit} and $\varepsilon_A : A \to K$ called {\it augmentation} such that $\varepsilon_A \circ \eta_A=\id_K$,
and a graded $K$-coderivation $d$ of degree $-1$ on the coaugmented tensor graded $K$-coalgebra $T(s\ol{A}) = \bigoplus\limits_{n\in\mathbb{N}_0}(s\ol{A})^{\ot n}$
where $\ol{A}:= \Ker(\varepsilon_A)$, such that $d\circ\eta_{T(s\ol{A})} = 0$ and $d^2 = 0$.
Let $1_A := \eta_A(1_K)$. Then $A = K1_A \oplus \ol{A}$ as graded $K$-bimodules. We always identify $K1_A$ with $K$.
The {\it bar construction} $BA$ of $A$ is the coaugmented dg $K$-coalgebra $T(s\ol{A})$ with differential $d$.
Since $BA$ is a cocomplete cofree coaugmented graded $K$-coalgebra,
the graded $K$-coderivation $d$ is uniquely determined by $\bar{d} := p_{s\ol{A}} \circ d$
where $p_{s\ol{A}} : BA \twoheadrightarrow s\ol{A}$ is the canonical projection.
Write $\bar{d} = \sum\limits_{n\in\mathbb{N}}\bar{d}_n$ where $\bar{d}_n : (s\ol{A})^{\ot n} \to s\ol{A}$.
For any $n \in \mathbb{N}$, define $\ol{m}_n = (-1)^ns^{-1} \circ \bar{d}_n \circ s^{\ot n} : \ol{A}^{\ot n} \to \ol{A}$,
and extend it to the graded $K$-bimodule morphism $m_n : A^{\ot n} \to A$
which is the composition $i\circ \ol{m}_n \circ p^{\ot n}$ if $n\ne 2$,
where $i:\ol{A}\hookrightarrow A$ and $p: A \twoheadrightarrow \ol{A}$ are the canonical inclusion and projection respectively,
and which is given by $m_2|_{\ol{A}^{\ot 2}} = \ol{m}_2$ and $m_2(\eta_A \ot \id_A) = \id_A = m_2(\id_A \ot \eta_A)$ if $n=2$.
An equivalent definition of augmented $A_\infty$-algebra is given by the data $(A,\{m_n\}_{n\in\mathbb{N}},\eta_A,\varepsilon_A)$
satisfying appropriate properties \cite{LuPalWuZha04}.

A {\it morphism} or {\it $A_\infty$-morphism} $f : A \to A'$ between two augmented $A_\infty$-algebras $A$ and $A'$
is a coaugmented dg $K$-coalgebra morphism $Bf : BA \to BA'$.
Since $BA'$ is a coaugmented tensor graded $K$-coalgebra,
the morphism $Bf$ is uniquely determined by its composition with the canonical projection $p'_{s\ol{A'}}: BA' \to s\ol{A'}$.
Write $F :=p'_{s\ol{A'}}\circ Bf = \sum\limits_{n\in\mathbb{N}} F_n$, where $F_n : (s\ol{A})^{\ot n} \to s\ol{A'}$.
For any $n \in\mathbb{N}$, define $\ol{f}_n = (-1)^{n-1}s^{-1} \circ F_n \circ s^{\ot n} : \ol{A}^{\ot n} \to \ol{A'}$,
and extend it to the map $f_n : A^{\ot n} \to A'$ which is the composition $i'\circ\ol{f}_n \circ p^{\ot n}$ if $n \ge 2$,
where $i':\ol{A'}\hookrightarrow A'$ and $p: A\twoheadrightarrow \ol{A}$ are the canonical inclusion and projection respectively,
and which is given by $f_1|_K = \id_K$ and $f_1|_{\ol{A}} = \ol{f}_1$ if $n=1$.
An equivalent definition of augmented $A_\infty$-algebra morphism is given by the family of maps $\{f_n\}_{n\in\mathbb{N}}$
satisfying appropriate properties \cite{LuPalWuZha04}.

Note that $f_1$ is a dg $K$-bimodule morphism from $(A,m_1)$ to $(A',m'_1)$.
An augmented $A_\infty$-algebra morphism $f:A\to A'$ is {\it strict} if $f_n=0$ for all $n\ge 2$,
and it is a {\it quasi-isomorphism} if $f_1: (A,m_1) \to (A',m'_1)$ is a dg $K$-bimodule quasi-isomorphism.
The {\it identity morphism} is the strict morphism $f$ with $f_1=\id$.

\bigskip

\noindent{\bf Coaugmented $A_\infty$-coalgebras.}
A {\it coaugmented $A_\infty$-coalgebra} $C=(C,\varepsilon_C,\eta_C, \linebreak d)$ is a graded $K$-bimodule $C$
together with two graded $K$-bimodule morphisms of degree zero
$\varepsilon_C : C \to K$ called {\it counit} and $\eta_C : K \to C$ called {\it coaugmentation} such that $\varepsilon_C \circ \eta_C=\id_K$,
and a graded $K$-derivation $d$ of degree $-1$ on the augmented tensor graded $K$-algebra $T(s^{-1}\ol{C})$ where $\ol{C} := \Ker(\varepsilon_C)$
such that $\varepsilon_{T(s^{-1}\ol{C})} \circ d = 0$ and $d^2 = 0$.
Let $1_C = \eta_C(1_K)$. Then $C = K 1_C \oplus \ol{C}$ as graded $K$-bimodules. We always identify $K1_C$ with $K$.
The {\it cobar construction} $\Omega C$ of $C$ is the augmented dg $K$-algebra $T(s^{-1}\ol{C})$ with differential $d$.
Since $\Omega C$ is a free augmented graded $K$-algebra, $d$ is uniquely determined by its restriction
to $s^{-1}\ol{C}$, which is denoted by $\bar{d} = \sum\limits_{n\in\mathbb{N}} \bar{d}_n$ where $\bar{d}_n : s^{-1}\ol{C} \to (s^{-1}\ol{C})^{\ot n}$.
For any $n\in\mathbb{N}$, define $\ol{\Delta}_n = (-1)^{\frac{n(n-1)}{2}-1}s^{\ot n}\circ \bar{d}_n \circ s^{-1}: \ol{C} \to \ol{C}^{\ot n}$,
and extend it to the $K$-bimodule morphism $\Delta_n : C \to C^{\ot n}$
which is the composition of the canonical projection $C \to \ol{C}$, $\ol{\Delta}_n$
and the injection $\ol{C}^{\ot n} \to C^{\ot n}$ if $n \ne 2$,
and which is given by $\Delta_2(1_C) = 1_C \ot 1_C$
and $\Delta_2(\bar{c}) = \ol{\Delta}_2(\bar{c}) + 1_C \ot \bar{c} + \bar{c}\ot 1_C$ for all $\bar{c}\in\ol{C}$ if $n=2$.
An equivalent definition of coaugmented $A_\infty$-coalgebra is given by the data $(C,\{\Delta_n\}_{n\in\mathbb{N}},\varepsilon_C,\eta_C)$
satisfying appropriate properties \cite{Her18HHA}.

A {\it morphism} or {\it $A_\infty$-morphism} $f:C\to C'$ between two coaugmented $A_\infty$-coalgebras $C$ and $C'$
is a morphism of augmented dg $K$-algebras $\Omega f :\Omega C\to \Omega C'$.
Since $\Omega C$ is a free augmented graded $K$-algebra, the morphism $\Omega f$ is uniquely determined
by its restriction to $s^{-1}\ol{C}$, which we denote by $F = \sum\limits_{n\in\mathbb{N}}F_n$,
where $F_n : s^{-1}\ol{C} \to (s^{-1}\ol{C'})^{\ot n}$.
For any $n\in\mathbb{N}$, define $\ol{f}_n = (-1)^{\frac{n(n-1)}{2}}s^{\ot n} \circ F_n\circ s^{-1} : \ol{C}\to \ol{C'}^{\ot n}$,
and extend to the map $f_n : C\to (C')^{\ot n}$ which is the composition of the canonical projection $C \to\ol{C}$, $\ol{f}_n$ and
the injection $\ol{C'}^{\ot n}\to (C')^{\ot n}$ if $n \ge 2$,
and which is given by $f_1|_K = \id_K$ and $f_1|_{\ol{C}} = \ol{f}_1$ if $n=1$.
An equivalent definition of coaugmented $A_\infty$-coalgebra morphism is given by the family of maps $\{f_n\}_{n\in\mathbb{N}}$
satisfying appropriate properties \cite{Her18HHA}.
Note that $f_1$ is a dg $K$-bimodule morphism from $(C,\Delta_1)$ to $(C',\Delta'_1)$.
A coaugmented $A_\infty$-coalgebra morphism $f:C\to C'$ is {\it strict} if $f_n=0$ for all $n\ge 2$,
and it is a {\it quasi-isomorphism} if $f_1: (C,\Delta_1) \to (C',\Delta'_1)$ is a dg $K$-bimodule quasi-isomorphism.

\bigskip

\noindent{\bf $A_\infty$-bimodules.}
Let $A$ be an augmented $A_\infty$-algebra. An {\it $A_\infty$-bimodule over $A$} or {\it $A_\infty$-$A$-bimodule}
$M=(M,d)$ is a graded $K$-bimodule $M$ and a graded $K$-bicoderivation $d$ of degree $-1$ on the graded $BA$-bicomodule $BA \ot M\ot BA$
compatible with the differential of $BA$ such that $d^2= 0$.
Denote by $\widehat{M}$ the dg $BA$-bicomodule $BA \ot M\ot BA$ with differential $d$.
Since $\widehat{M}$ is a cofree graded bicomodule, the graded $K$-bicoderivation $d$ is uniquely determined by its composition
with $\varepsilon_{BA}\ot\id_{M}\ot\varepsilon_{BA}$,
which is denoted by $\bar{d}=\sum\limits_{p,q\in\mathbb{N}_0}\bar{d}_{p,q}$
where $\bar{d}_{p,q} : (s\ol{A})^{\ot p}\ot M\ot(s\ol{A})^{\ot q}\to M$ for all $p,q\in\mathbb{N}_0$.
Define $m_{p,q} : A^{\ot p}\ot M\ot A^{\ot q}\to M$ which is the composition of the canonical projection
$A^{\ot p}\ot M\ot A^{\ot q} \to \ol{A}^{\ot p}\ot M\ot\ol{A}^{\ot q}$ and
$(-1)^q\ \bar{d}_{p,q} \circ (s^{\ot p}\ot \id_M\ot s^{\ot q})$  if $p+q\ne 1$,
and which is given by $m_{0,1}\circ(\id_M \ot \eta_A)=\id_M =m_{1,0}\circ (\eta_A \ot \id_M)$,
$m_{0,1}\circ (\id_M \ot i_A) = -\bar{d}_{0,1}\circ (\id_M \ot s)$ and
$m_{1,0}\circ (i_A \ot \id_M) = \bar{d}_{1,0}\circ (s \ot \id_M)$ if $p+q=1$,
where $i_A : \ol{A}\hookrightarrow A$ is the canonical inclusion.
An equivalent definition of $A_\infty$-bimodule is given by the data $(M,\{m_{p,q}\}_{p,q\in\mathbb{N}_0})$
satisfying appropriate properties.

Let $A$ be an augmented $A_\infty$-algebra defined by operators $\{m_n: A^{\ot n}\to A\}_{n\in\mathbb{N}}$.
Then $A$ itself is an $A_\infty$-$A$-bimodule defined by the operators $\{m_{p,q}: A^{\ot p}\ot A\ot A^{\ot q} \to A\}_{p,q\in\mathbb{N}_0}$ given by
$m_{p,q}(a_1\ot\cdots\ot a_p\ot a_{p+1}\ot a_{p+2}\ot\cdots\ot a_{p+q+1}):= m_{p+q+1}(a_1\ot\cdots\ot a_{p+q+1})$.

A {\it morphism of $A_\infty$-bimodules $f:M\to N$} between two $A_\infty$-bimodules $M$ and $N$
is a morphism of dg $BA$-bicomodules $\widehat{f} : \widehat{M}\to \widehat{N}$.
Since $\widehat{N}$ is a cofree graded bicomodule, $\widehat{f}$ is uniquely
determined by its composition with $\varepsilon_{BA} \ot\id_N\ot \varepsilon_{BA}$,
which is written as $F=\sum\limits_{p,q\in\mathbb{N}_0}F_{p,q}$ where $F_{p,q} : (s\ol{A})^{\ot p} \ot M\ot(s\ol{A})^{\ot q} \to N$.
Define $f_{p,q} : A^{\ot p}\ot M\ot A^{\ot q}\to N$ as the composition of the canonical projection
$A^{\ot p}\ot M\ot A^{\ot q} \to \ol{A}^{\ot p}\ot M\ot \ol{A}^{\ot q}$ and $(-1)^qF_{p,q}\circ(s^{\ot p}\ot \id_M\ot s^{\ot q})$.
An equivalent definition of $A_\infty$-bimodule morphism is given by the family of maps $\{f_{p,q}\}_{p,q\in\mathbb{N}_0}$
satisfying appropriate properties.
An $A_\infty$-bimodule morphism $f$ is {\it strict} if $f_{p,q}=0$ for all $(p,q)\ne(0,0)$.

Similarly, one can define left (resp. right) $A_\infty$-modules over an augmented $A_\infty$-algebra $A$.
They correspond to left (resp. right) dg $BA$-comodules.

The following result is well-known.

\begin{lemma}\label{Lemma-A-bimod-iso}
Let $A$ be an augmented $A_\infty$-algebra and $f:M\to N$ an $A_\infty$-$A$-bimodule morphism.
Then $\widehat{f}: \widehat{M}\rightarrow \widehat{N}$ is a dg $BA$-bicomodule isomorphism
if and only if $f_{0,0}$ is an isomorphism of dg $K$-bimodules.
\end{lemma}

\begin{proof}
{\it Necessity.} Assume that $g:N\to M$ is an $A_\infty$-$A$-bimodule morphism
such that $\widehat{g}:\widehat{N}\rightarrow \widehat{M}$ is the inverse of $\widehat{f}$.
Then $\widehat{f}\circ\widehat{g}=\mathrm{id}_{\widehat{N}}$ and $\widehat{g}\circ\widehat{f}=\mathrm{id}_{\widehat{M}}$.
Thus the families $\{F_{p,q}\}_{p,q \in \mathbb{N}_0}$ and $\{G_{p,q}\}_{p,q\in \mathbb{N}_0}$ satisfy $F_{0,0}\circ G_{0,0}=\mathrm{id}_{N}$,
$$\sum\limits_{0\leq i\leq p, 0\leq j\leq q}F_{i,j}\circ (\mathrm{id}^{\ot i}\ot G_{p-i,q-j}\ot \mathrm{id}^{\ot j})=0$$
for all $(p,q)\ne (0,0)$, $G_{0,0}\circ F_{0,0}=\mathrm{id}_{M}$, and
$$\sum\limits_{0\leq i\leq p, 0\leq j\leq q}G_{i,j}\circ (\mathrm{id}^{\ot i}\ot F_{p-i,q-j}\ot \mathrm{id}^{\ot j})=0$$
for all $(p,q)\ne (0,0)$. In particular, $f_{0,0}=F_{0,0}$ is an isomorphism of dg $K$-bimodules.

{\it Sufficiency.} Assume that $f_{0,0}$ is an isomorphism of dg $K$-bimodules and $g_{0,0}: N\rightarrow M$ is its inverse.
Let $G_{0,0}=g_{0,0}$. By the equation
$$F_{0,0}\circ G_{p,q}+\sum\limits_{0\leq i\leq p,\ 0\leq j\leq q,\ i+j>0}F_{i,j}\circ (\mathrm{id}^{\ot i}\ot G_{p-i,q-j}\ot \mathrm{id}^{\ot j})=0,$$
we have
$$G_{p,q}=- G_{0,0}\circ(\sum\limits_{0\leq i\leq p,\ 0\leq j\leq q,\ i+j>0}F_{i,j}\circ (\mathrm{id}^{\ot i}\ot G_{p-i,q-j}\ot \mathrm{id}^{\ot j})). $$
Thus we can construct the family of graded $K$-bimodules $\{G_{p,q}\}_{p,q\in \mathbb{N}_{0}}$ inductively.
The family $\{G_{p,q}\}_{p,q\in \mathbb{N}_0}$ gives a graded $BA$-bicomodule morphism $\widehat{g}: \widehat{N}\rightarrow \widehat{M}$
and it satisfies $\widehat{f}\circ \widehat{g}=\mathrm{id}_{\widehat{N}}$.
Similarly, we can construct a graded $BA$-bicomodule morphism $\widehat{g^{\prime}}: \widehat{N}\rightarrow \widehat{M}$
such that $\widehat{g^\prime}\circ \widehat{f}=\mathrm{id}_{\widehat{M}}$,
and we have $\widehat{g}=\widehat{g^\prime}\circ \widehat{f}\circ \widehat{g}=\widehat{g^\prime}$.
Thus $\widehat{g}$ is the inverse of the graded $BA$-bicomodule morphism $\widehat{f}$ and it is compatible with the differential naturally.
Therefore, $\widehat{g}$ is the inverse of the dg $BA$-bicomodule morphism $\widehat{f}$.
\end{proof}

\bigskip

\noindent {\bf Two-sided bar complexes.}
Let $A$ be an augmented $A_\infty$-algebra. Then $\wt{A}:=A\ot BA\ot A$ is an $A_\infty$-$A$-bimodule
which is defined by the differential $d$ on the graded $BA$-bicomodule $BA\ot \wt{A}\ot BA=BA\ot A\ot BA\ot A\ot BA$
$$\begin{array}{ll}
d:=& d_{BA}\ot\id^{\ot 4}+\id^{\ot 2}\ot d_{BA}\ot\id^{\ot 2} +\id^{\ot 4}\ot d_{BA} \vspace{2mm}\\
& +(\id_{BA}\ot \bar{d}\ot\id_{BA}\ot\id_A\ot\id_{BA})\circ(\Delta_{BA}\ot\id_A\ot\Delta_{BA}\ot\id_A\ot\id_{BA}) \vspace{2mm}\\
& +(\id_{BA}\ot\id_A\ot\id_{BA}\ot\bar{d}\ot\id_{BA})\circ(\id_{BA}\ot\id_A\ot\Delta_{BA}\ot\id_A\ot\Delta_{BA}).
\end{array}$$
Moreover, there is an $A_\infty$-$A$-bimodule quasi-isomorphism $\wt{\mu}: \wt{A} \to A$
defined by the maps $\{\wt{\mu}_{p,q}: A^{\ot p}\ot \wt{A} \ot A^{\ot q} \to A\}_{p,q\in \mathbb{N}_0}$,
where the restriction of $\wt{\mu}_{p,q}$ on $A^{\ot p}\ot A\ot (s\overline{A})^{\ot r}\ot A\ot A^{\ot q}$ is given by
$(-1)^{p+rq+\frac{r(r-1)}{2}}m_{p+q+r+2}\circ (\mathrm{id}^{\ot p+1}\ot (s^{-1})^{\ot r}\ot \mathrm{id}^{\ot q+1})$.
For this, we need to show that the $K$-bimodule complex morphism $\wt{\mu}_{0,0}:(\wt{A}, m_{0,0}^{\wt{A}})\to (A, m_{0,0}^A=m_1)$ is a quasi-isomorphism.
Define a $k$-module complex morphism $\nu: A\to \wt{A}$ by $\nu(a):=1_A\ot 1_{BA}\ot a$,
and a graded $k$-module morphism of degree one $s: \wt{A}\to \wt{A}$ by $s(a_0\ot [a_1|\cdots|a_n]\ot a_{n+1}):=1_A\ot[a_0|a_1|\cdots|a_n]\ot a_{n+1}$.
Then one can check that $\wt{\mu}_{0,0}\circ \nu=\id_A$ and $\nu\circ\wt{\mu}_{0,0}-\id_{\wt{A}}=m_{0,0}^{\wt{A}}\circ s+s\circ m_{0,0}^{\wt{A}}$,
i.e., $\wt{\mu}_{0,0}$ is a $k$-module complex, and thus a $K$-bimodule complex, quasi-isomorphism.

\bigskip

\noindent {\bf Suspensions.} (\cite[2.3.2]{Her18})
Let $M$ be an $A_\infty$-bimodule over an augmented $A_\infty$-algebra $A$.
The {\it suspension} $sM$ or $M[1]$ of $M$ is the $A_\infty$-$A$-bimodule
which is defined by the differential $d$ on the graded $BA$-bicomodule $BA\ot sM\ot BA$
uniquely determined by the following commutative diagram:
$$\xymatrixcolsep{6pc}\xymatrixrowsep{3pc}\xymatrix{
BA\ot M\ot BA \ar[r]^-{d_{\widehat{M}}} \ar[d]^-{\id_{BA}\ot s\ot \id_{BA}}
& BA\ot M\ot BA \ar[r]^-{\varepsilon_{BA}\ot\id_M\ot\varepsilon_{BA}} \ar[d]^-{\id_{BA}\ot s\ot \id_{BA}} & M \ar[d]^-s \\
BA\ot sM\ot BA \ar[r]^-{-d} & BA\ot sM\ot BA \ar[r]^-{\varepsilon_{BA}\ot\id_{sM}\ot\varepsilon_{BA}} & sM
}$$

\bigskip

\noindent{\bf Graded duals.} (cf. \cite[2.3.2]{Her18})
Let $M$ be an $A_\infty$-bimodule over an augmented $A_\infty$-algebra $A$.
Then the {\it graded dual} $M^\vee$ of $M$ is the $A_\infty$-$A$-bimodule
which is defined by the differential $d$ on the graded $BA$-bicomodule $BA\ot M^\vee\ot BA$
uniquely determined by the following commutative diagram:

{\footnotesize $$\xymatrixcolsep{5.7pc}\xymatrixrowsep{2pc}\xymatrix{
M^\vee\ot BA\ot M\ot BA \ar[r]^-{\id_{M^\vee}\ot d_{\widehat{M}}} \ar[dd]^-{S_{4123}}
& M^\vee\ot BA\ot M\ot BA \ar[r]^-{\id_{M^\vee}\ot \varepsilon_{BA}\ot\id_M\ot\varepsilon_{BA}} \ar[dd]^-{S_{4123}} & M^\vee\ot M \ar[d]_-{\mathrm{ev}_M} \\
&&k\\
BA\ot M^\vee\ot BA\ot M \ar[r]^-{-d\ot\id_M} & BA\ot M^\vee\ot BA\ot M \ar[r]^-{\varepsilon_{BA}\ot\id_{M^\vee}\ot\varepsilon_{BA}\ot\id_M}
& M^\vee\ot M \ar[u]^-{\mathrm{ev}_M}
}$$ }
where $\mathrm{ev}_M: M^\vee\ot M\to k$ is the usual evaluation map.

\bigskip

\noindent{\bf Koszul duals.}
An {\it $A$-finite $A_\infty$-algebra} $A$ is a locally finite, bounded below or bounded above, augmented $A_\infty$-algebra $A$ {\it of finite products},
i.e., $m_n=0$ for $n\gg 0$, or equivalently, the composition $\bar{d}:=p_{s\ol{A}}\circ d_{BA}: BA\xrightarrow{d_{BA}} BA\xrightarrow{p_{s\ol{A}}} s\ol{A}$
satisfies $\bar{d}_n=0$ for $n\gg 0$.
An {\it $A$-finite $A_\infty$-coalgebra} $C$ is a locally finite, bounded above or bounded below, coaugmented $A_\infty$-coalgebra $C$ {\it of finite coproducts},
i.e., $\Delta_n=0$ for $n\gg 0$, or equivalently, the composition $\bar{d}:=d_{\Omega C} \circ i_{s^{-1}\ol{C}}:
s^{-1}\ol{C} \xrightarrow{i_{s^{-1}\ol{C}}} \Omega C\xrightarrow{d_{\Omega C}} \Omega C$ satisfies $\bar{d}_n=0$ for $n\gg 0$.

Let $(A,\{m_n\}_{n\in\mathbb{N}},\eta_A,\varepsilon_A)$ be an $A$-finite augmented $A_\infty$-algebra.
Then we have graded $K$-bimodule isomorphism $(A^{\ot n})^\vee \cong (A^\vee)^{\ot n}$ for all $n\in \mathbb{N}$,
Thus the graded dual $(A^\vee,\{\Delta_n\},\varepsilon_{A^\vee},\eta_{A^\vee})$ of $A$ is an $A$-finite $A_\infty$-coalgebra
where $\Delta_n=(-1)^nm_n^\vee,\ \eta_{A^\vee}=(\varepsilon_A)^\vee$ and $\varepsilon_{A^\vee}={\eta_A^\vee}$.
The {\it Koszul dual} of an $A$-finite $A_\infty$-algebra $A$ is $A^\dag := \Omega(A^\vee)$.

\begin{remark}{\rm
Let $A$ be an augmented $A_\infty$-algebra.
The {\it Koszul dual} of $A$ is usually defined to be the augmented dg $K$-algebra $(BA)^\vee$ (see \cite{LuPalWuZha08}).
This definition of Koszul dual coincides with ours in many important situations.
An augmented $A_\infty$-algebra $A$ is {\it typical} if $A$ is a locally finite augmented $A_\infty$-algebra
which is either non-negative or non-positive simply connected.
If $A$ is a typical $A_\infty$-algebra of finite products then $A^\vee$ is a coaugmented $A_\infty$-coalgebra and $(BA)^\vee \cong \Omega(A^\vee)$.
}\end{remark}

\section{Hochschild extensions}

In this section, we introduce the Hochschild extensions of an augmented dg $K$-algebra by a dg bimodule and a Hochschild 2-cocycle.
Furthermore, we focus on the Hochschild extensions of an augmented dg $K$-algebra by a shift of its graded dual dg bimodule and a Hochschild 2-cocycles,
and show that exact Hochschild extensions are symmetric Hochschild extensions.

\subsection{Hochschild (co)homology}

Hochschild (co)homology is necessary for studying Hochschild extensions.

\bigskip

\noindent {\bf Hochschild cohomology of dg $K$-algebras.}
Let $A$ be an augmented dg $K$-algebra and $M$ a dg $A$-bimodule.
The {\it Hochschild cochain complex of $A$ with coefficients in $M$} is
$C^\bullet(A,M) := \Hom_{K^e}(BA,M) \cong \Hom_{A^e}(A \ot BA\ot A, M) \cong \RHom_{A^e}(A,M)$.
Its cohomology $H^\bullet(A,M)$ is called the {\it Hochschild cohomology of $A$ with coefficients in $M$}.
Note that the bullet $\bullet$ indicates weight in $C^\bullet(A,M)$ and degree in $H^\bullet(A,M)$.
More precisely, $C^\bullet(A,M) = \prod\limits_{n \in \mathbb{N}_0} C^n(A,M)$ where $C^n(A,M) := \Hom_{K^e}((s\ol{A})^{\ot n},M)$,
and the differential of $C^\bullet(A,M)$ is $\dz = \dz_0+\dz_1$ where $\dz_0$ is the {\it inner differential} given by
$$\dz_0(f)[a_1|\cdots|a_n] := d_Mf[a_1|\cdots|a_n] + \sum_{i=1}^n(-1)^{\varepsilon_{i-1}+|f|}\ f[a_1|\cdots|d_Aa_i|\cdots|a_n],$$
and $\dz_1$ is the {\it external differential} given by
$$\begin{array}{ll}
\dz_1(f)[a_1|\cdots|a_{n+1}] := & (-1)^{|f|(|a_1|+1)}\ a_1f[a_2|\cdots|a_{n+1}] \vspace{2mm}\\
& \quad + \sum\limits_{i=1}^n(-1)^{|f|+\varepsilon_i}\ f[a_1|\cdots|a_ia_{i+1}|\cdots|a_{n+1}] \vspace{2mm}\\
& \quad\quad + (-1)^{|f|+\varepsilon_n+1}\ f[a_1|\cdots|a_n]a_{n+1},
\end{array}$$
for all $f \in C^n(A,M)$.
Here $\varepsilon_i := \sum\limits_{j=1}^{i}(|a_j|+1)$.

In the case of $M=A$, $C^\bullet(A,A)$ is called the {\it Hochschild cochain complex of $A$} and denoted by $C^\bullet(A)$.
Its cohomology $H^\bullet(A,A)$ is called the {\it Hochschild cohomology of $A$}, denoted by $HH^\bullet(A)$.

\bigskip

\noindent {\bf Hochschild homology of dg $K$-algebras.}
Let $A$ be an augmented dg $K$-algebra and $M$ a dg $A$-bimodule.
The {\it Hochschild chain complex of $A$ with coefficients in $M$} is
$C_\bullet(A,M) := M\ot_{K^e}BA \cong M \ot_{A^e}(A\ot BA\ot A) \cong M \ot^L_{A^e} A$.
Its homology $H_\bullet(A,M)$ is called the {\it Hochschild homology of $A$ with coefficients in $M$}.
Note that the bullet $\bullet$ indicates weight in $C_\bullet(A,M)$ and degree in $H_\bullet(A,M)$.
More precisely, $C_\bullet(A,M) = \bigoplus\limits_{n \in \mathbb{N}_0} C_n(A,M)$ with $C_n(A,M) := M \ot _{K^e}(s\ol{A})^{\ot n}$,
and the differential of $C_\bullet(A,M)$ is $b := b_0 + b_1$ where $b_0$ is the {\it inner differential} given by
$$b_0(m \ot [a_1|\cdots|a_n]) := d_Mm \ot [a_1|\cdots|a_n] + \sum_{i=1}^n(-1)^{\eta_{i-1}}\ m \ot [a_1|\cdots|d_Aa_i|\cdots|a_n],$$
and $b_1$ is the {\it external differential} given by:
$$\begin{array}{ll}
b_1(m \ot [a_1|\cdots|a_n]) := & (-1)^{|m|+1}\ ma_1 \ot [a_2|\cdots|a_n] \vspace{2mm}\\
& \quad + \sum\limits_{i=1}^{n-1}(-1)^{\eta_i}\ m \ot [a_1|\cdots|a_ia_{i+1}|\cdots|a_n] \vspace{2mm}\\
& \quad\quad + (-1)^{(\eta_{n-1}+1)(|a_n|+1)}\ a_nm \ot [a_1|\cdots|a_{n-1}]
\end{array}$$
where $\eta_i := |m|+1+\sum\limits_{j=1}^i(|a_j|+1)$.

In the case of $M=A$, $C_\bullet(A,A)$ is called the {\it Hochschild chain complex of $A$} and denoted by $C_\bullet(A)$.
Its cohomology $H_\bullet(A,A)$ is called the {\it Hochschild homology of $A$}, denoted by $HH_\bullet(A)$.

\bigskip

\noindent{\bf Hochschild homology of $A_\infty$-algebras.}
Let $A$ be an augmented $A_\infty$-algebra and $M$ an $A_\infty$-$A$-bimodule.
The {\it Hochschild chain complex of $A$ with coefficients in $M$} is the complex $C_\bullet(A,M) := M\ot_{K^e}BA$ endowed with differential $d=d_1+d_2$,
where $d_1=\id_M\ot d_{BA}$ and $d_2$ is the composition
$M\ot_{K^e}BA\xrightarrow{\id_M\ot \Delta^{(2)}}M\ot_{K^e} (BA)^{\ot 3}\xrightarrow{S_{4123}}\widehat{M}\ot_{K^e}BA\xrightarrow{\bar{d} \ot \id_M}M\ot_{K^e}BA$
where $\bar{d}=(\varepsilon_{BA}\ot\id_M\ot \varepsilon_{BA})\circ d_{\widehat{M}}$
and $S_{4123} : V_1\ot V_2\ot V_3\ot V_4 \to V_4 \ot V_1\ot V_2\ot V_3,\ v_1\ot v_2\ot v_3\ot v_4 \mt (-1)^{|v_4|(|v_1|+|v_2|+|v_3|)}v_4 \ot v_1\ot v_2\ot v_3$.
Its homology $H_\bullet(A,M)$ is called the {\it Hochschild homology of $A$ with coefficients in $M$}.
In the case of $M=A$, $C_\bullet(A,A)$ is called the {\it Hochschild chain complex} of the $A_\infty$-algebra $A$ and denoted by $C_\bullet(A)$.
Its homology $H_\bullet(A,A)$ is called the {\it Hochschild homology of $A$} and denoted by $HH_\bullet(A)$.

An $A_\infty$-$A$-bimodule morphism $f:M\rightarrow N$ induces a dg $k$-module morphism $C_\bullet(A,f): C_\bullet(A,M)\rightarrow C_\bullet(A,N)$
which is the composition
$M\ot_{K^e}BA \xrightarrow{\id_M\ot \Delta^{(2)}} M\ot_{K^e}(BA)^{\ot 3} \xrightarrow{S_{4123}} \widehat{M}\ot_{K^e}BA
\xrightarrow{\widehat{f}\ot \id_{BA}} \widehat{N}\ot_{K^e}BA \xrightarrow{\varepsilon_{BA}\ot\id_N\ot\varepsilon_{BA}\ot \id_{BA}} N\ot_{K^e}BA$.

\bigskip

By the same proof as \cite[Theorem 2.10]{Mes16}, we can obtain the following result:

\begin{lemma} \label{Lemma-A-Inf-Iso-0} {\rm(\cite[Lemma 5.3]{GetJon90} and \cite[Theorem 2.10]{Mes16})}
Let $A$ be an augmented $A_\infty$-algebra and $f: M\rightarrow N$ an $A_\infty$-$A$-bimodule morphism.
If the $K$-bimodule complex morphism $f_{0,0}: (M, m_{0,0}^M)\rightarrow (N, m_{0,0}^N)$ is a quasi-isomorphism,
then the induced $k$-module complex morphism $C_{\bullet}(A,f):C_{\bullet}(A,M)\rightarrow C_{\bullet}(A,N)$ is a quasi-isomorphism.
\end{lemma}

\bigskip

Let $A$ be an augmented $A_\infty$-algebra. Then we have an $A_\infty$-bimodule quasi-isomorphism $\wt{\mu}: \wt{A}\rightarrow A$.
It follows from Lemma \ref{Lemma-A-Inf-Iso-0} that $k$-module complex morphism
$C_\bullet(A,\wt{\mu}): C_\bullet(A,\wt{A})\rightarrow C_\bullet(A)$ is a quasi-isomorphism.
Here, $C_\bullet(A,\wt{A})$ is the complex $\wt{A}\ot_{K^e} BA$.
When $A$ is an augmented dg $K$-algebra, $C_\bullet(A,\wt{A})$ is just $\wt{A}\ot_{A^e}\wt{A}$,
the double replacement of $A\ot^L_{A^e}A$ with two-sided bar resolution.
Moreover, $C_\bullet(A,\wt{\mu})$ has a quasi-inverse $\delta: A\ot_{K^e}BA\rightarrow \wt{A}\ot_{K^e}BA$
which is given by $\delta(a\ot [a_1|\cdots|a_n])=\sum\limits_{i=0}^n a\ot [a_1|\cdots|a_i]\ot 1_A\ot [a_{i+1}|\cdots|a_n]$
and satisfies $C_\bullet(A,\wt{\mu})\circ \delta =\mathrm{id}_{C_{\bullet}(A)}$.

\begin{lemma} \label{Lemma-Free-Bicom-Hom}
Let $A$ be an augmented $A_\infty$-algebra. Then there is a dg $k$-module  isomorphism
$$\Hom_{(BA)^e}(\widehat{A}, \widehat{A^\vee})\cong C_{\bullet}(A,\wt{A})^\vee.$$
\end{lemma}

\begin{proof}
It is clear that $\Hom_{(BA)^e}(\widehat{A},\widehat{A^\vee}) \cong \Hom_{K^e}(\widehat{A},A^\vee)
\cong (A\otimes_{K^e}\widehat{A})^\vee \cong ((A\ot BA\ot A)\ot_{K^e}BA)^\vee =C_{\bullet}(A,\wt{A})^\vee$ as graded $k$-modules.
We can check that the composition is not only a graded $k$-module isomorphism but also compatible with differentials, and thus a dg $k$-module  isomorphism.
\end{proof}

\bigskip

\noindent {\bf Connes operators.} Let $A$ be an augmented $A_\infty$-algebra.
The {\it Connes operator} $B$ on the Hochschild chain complex $C_\bullet(A)=A\ot_{K^e} BA$ is defined by
$$B(a_0 \ot [a_1|\cdots|a_n]) := \sum_{i=0}^n (-1)^{\eta_i(\eta_n-\eta_i)}\ 1 \ot [a_{i+1}|\cdots|a_n|a_0|\cdots|a_i],$$
where $\eta_i := \sum\limits_{j=0}^i(|a_j|+1)$ (see \cite{GetJon90}). It satisfies $B^2 =0$ and $Bb+bB=0$.
Therefore, $C_\bullet(A)$ is a mixed complex.
Let $\Lambda$ be the dg algebra $k[\epsilon]/(\epsilon^2)$ with $|\epsilon|=1$ and differential zero.
Then $C_{\bullet}(A)$ becomes a dg $\Lambda$-module with $\epsilon$ acting on $C_{\bullet}(A)$ by the Connes operator $B$.

The Connes operator $B:A\ot_{K^e} BA \rightarrow A\ot_{K^e} BA$ can be lifted to an operator
$$\wt{B}: \wt{A}\ot_{K^e} BA \rightarrow A\ot_{K^e} BA$$
mapping $a\ot [a_1|\cdots|a_n]\ot a'\ot[a'_1|\cdots|a'_m]$ to
$(-1)^{\varepsilon_n+|a|}a\ot [a_1|\cdots|a_n|a'|a'_1|\cdots|a'_m]\linebreak +(-1)^{(\varepsilon_{m}'+|a'|)(\varepsilon_n+|a|+1)}a'\ot[a'_1|\cdots|a'_m|a|a_1|\cdots|a_n]$,
where $\varepsilon_m'=\sum\limits_{i=1}^m(|a'_i|+1)$.
Moreover, the following diagram is commutative:
$$\xymatrix{
A\ot_{K^e} BA \ar[r]^-{\delta} \ar[dr]_-{B} & \wt{A}\ot_{K^e} BA \ar@{.>}[d]^-{\wt{B}} \\
& A\ot_{K^e} BA.
}$$

\subsection{Cyclic homology}

Cyclic homology is indispensable for studying exact Hochschild extensions.

\bigskip

\noindent{\bf Cyclic homology.} Let $A$ be an augmented dg $K$-algebra and $C_{\bullet}(A)$ the Hochschild chain complex of $A$
which is a mixed complex with Hochschild boundary operator $b$ of degree $-1$ and Connes operator $B$ of degree $1$.
Let $u$ be an indeterminant of degree $-2$, $k[[u]]$ the formal series algebra in $u$
which is a pseudo-compact graded algebra \cite{Bru66},
and $k((u))$ the fraction field of $k[[u]]$ or equivalently the Laurent series algebra $k[u,u^{-1}]$ in $u$.
The {\it negative cyclic complex} of $A$ is $CN_\bullet(A) := C_\bullet(A)[[u]]$,
the {\it periodic cyclic complex} of $A$ is $CP_\bullet(A) := C_\bullet(A)[[u]] \hat{\ot}_{k[[u]]} k((u))$,
and the {\it cyclic complex} of $A$ is $CC_\bullet(A) := C_\bullet(A)[[u]] \hat{\ot}_{k[[u]]} (k((u))/uk[[u]])$,
all with differential $b+uB$, where $``\hat{\ot}" $ stands for complete tensor product.
The homologies of $CN_\bullet(A), CP_\bullet(A)$ and $CC_\bullet(A)$
are {\it negative cyclic homology} $HN_\bullet(A)$, {\it periodic cyclic homology} $HP_\bullet(A)$ and {\it cyclic homology} $HC_\bullet(A)$,
of $A$ respectively.

Acting the functor $C_\bullet(A)[[u]] \hat{\ot}_{k[[u]]} -$ on the commutative diagram
$$\xymatrix{ 0 \ar[r] & k[[u]] \ar[r] \ar[d] & k((u)) \ar[r] \ar[d] & k((u))/k[[u]] \ar[r] \ar@{=}[d] & 0 \\
0 \ar[r] & k \ar[r] & k((u))/uk[[u]] \ar[r] & k((u))/k[[u]] \ar[r] & 0}$$
of $k[[u]]$-modules, where $k\cong k[[u]]/uk[[u]]$ and two rows are short exact sequences,
we get the following commutative diagram
$$\xymatrix{ 0 \ar[r] & CC_\bullet^-(A) \ar[r] \ar[d] & CP_\bullet(A) \ar[r] \ar[d] & CC_\bullet(A)[2] \ar[r] \ar@{=}[d] & 0 \\
0 \ar[r] & C_\bullet(A) \ar[r] & CC_\bullet(A) \ar[r] & CC_\bullet(A)[2] \ar[r] & 0.}$$
with exact rows. By taking homology, we obtain the following result:

\begin{lemma} \label{Lemma-Hoch-Cyc-LongExSeq}
Let $A$ be an augmented dg $K$-algebra. Then the following diagram
$$\xymatrixcolsep{1.8pc}\xymatrixrowsep{2pc}\xymatrix{
\cdots \ar[r] & HC_{n-1}(A) \ar[r]^-{B'_{n+1}} \ar@{=}[d] & HN_n(A) \ar[r]^-{I_n'} \ar[d]^-{P_n} &
HP_n(A) \ar[r]^-{S_n'} \ar[d] & HC_{n-2}(A) \ar[r] \ar@{=}[d] & \cdots \\
\cdots \ar[r] & HC_{n-1}(A) \ar[r]^-{B_{n+1}}  & HH_n(A) \ar[r]^-{I_n} & HC_n(A) \ar[r]^-{S_n} & HC_{n-2}(A) \ar[r] & \cdots
}$$
is commutative and with exact rows.
\end{lemma}
%\begin{remark}[\cite{Her18}] The cyclic homology and negative cyclic homology of dg $K$-algebra $A$ can also be interpreted the homologies of $C_{\bullet}(A)\otimes^{L}_{\Lambda}k$  and  $\RHom_{\Lambda}(k, C_{\bullet}(A))$ respectively.
%\end{remark}
\bigskip

\noindent{\bf (Almost) Exact Hochschild homology classes.} Some special Hochschild homology classes played quite important roles in Calabi-Yau algebra theory \cite{Van15,Her18,dTdVVdB18}.

\begin{definition}{\rm
Let $A$ be an augmented dg $K$-algebra.
A Hochschild homology class $[\az] \in HH_n(A)$ is {\it exact} if  $[\az] \in \Im B_{n+1}= \Ker I_n$, and
{\it almost exact} if $[\az] \in \Im P_n$.
}\end{definition}

\bigskip

\noindent{\bf Exact Hochschild cohomology classes.} Some special Hochschild cohomology classes will play an important role in symmetric Hochschild extension theory.
Let $A$ be an augmented dg $K$-algebra. Acting the exact functor $(-)^\vee=\Hom_k(-,k)$ on the Connes' long exact sequence
$$\cdots \rightarrow HH_{n}(A) \xrightarrow{I_n} HC_{n}(A) \xrightarrow{S_n} HC_{n-2}(A)\xrightarrow{B_n}HH_{n-1}(A) \rightarrow \cdots$$
we get the following long exact sequence:
$$\xymatrix{ \cdots & HC_{n-1}(A)^\vee \ar[l]  & HH_n(A)^\vee \ar[l]_-{B_{n+1}^\vee} & HC_n(A)^\vee \ar[l]_-{I_n^\vee} & \cdots \ar[l] }$$
Due to isomorphisms $H^n(A,A^\vee) = H^n\RHom_{A^e}(A,A^\vee) \cong H^n((A\ot^L_{A^e}A)^\vee) \cong (H_n(A\ot^L_{A^e}A))^\vee = HH_n(A)^\vee$,
we can identify $H^n(A,A^\vee)$ with $HH_n(A)^\vee$.

\bigskip

\begin{definition}{\rm
Let $A$ be an augmented dg $K$-algebra.
A Hochschild cohomology class $[\az] \in H^n(A,A^\vee)$ is {\it exact} if $[\az] \in \Im I_n^\vee=\Ker B_{n+1}^\vee$.
}\end{definition}

\subsection{Hochschild extensions}

In the classical case, for an ordinary algebra $A$, an $A$-bimodule $M$ and a Hochschild 2-cocycle $\az:A\ot A\to M$,
the Hochschild extension $\T (A,M,\az)$ of $A$ by $M$ and $\az$ is still an ordinary algebra (see for example \cite[1.5.3]{Lod98}).
In the dg case, for a dg $K$-algebra $A$ and a dg $A$-bimodule $M$,
the {\it trivial extension} $\T (A,M)$ of $A$ by $M$, i.e., $A\oplus M$
with the product given by $(a,m)\cdot(a',m')=(aa',am'+ma')$ for all $a,a'\in A,\ m,m'\in M$, is still a dg $K$-algebra.
However, for a Hochschild 2-cocycle $\az$ of $A$ with coefficients in $M$,
the Hochschild extension $\T (A,M,\az)$ of $A$ by $M$ and $\az$ is an $A_\infty$-algebra in general.

\bigskip

\noindent{\bf Hochschild extensions.}
Let $A$ be an augmented dg $K$-algebra, $M$ a dg $A$-bimodule with degree $-2$ component zero, i.e., $M_{-2}=0$,
and $\az\in C^\bullet(A,M)_{-2}=\Hom_{K^e}(BA,M)_{-2}$ a Hochschild cochain of degree $-2$.
To define an augmented $A_\infty$-algebra structure on $A\oplus M$ is equivalent to
define a graded $K$-coderivation $d$ on graded $K$-coalgebra $T(s\bar{A}\oplus sM)$ of degree $-1$ such that $d\circ\eta_{T(s\bar{A}\oplus sM)}=0$ and $d^2=0$,
and further equivalent to define a graded $K$-bimodule morphism $\bar{d}:T(s\ol{A}\oplus sM) \to s\ol{A}\oplus sM$ of degree $-1$.
Note that graded $K$-bimodule $T(s\ol{A}\oplus sM)=BA\oplus (BA\ot sM\ot BA)\oplus (BA\ot sM\ot BA\ot sM\ot BA)\oplus\cdots$.
We define $\bar{d}$ to be the composition
$$T(s\ol{A}\oplus sM)\twoheadrightarrow BA\oplus BA\ot sM\ot BA
\xrightarrow{\begin{bmatrix} p_{s\ol{A}}\circ d_{BA}&0\\-s\circ\alpha &p_{sM}\circ d_{BA\ot sM\ot BA}\end{bmatrix}} s\ol{A} \oplus sM$$
where $T(s\ol{A}\oplus sM)\twoheadrightarrow BA\oplus BA\ot sM\ot BA$,
$p_{s\ol{A}}:BA\to s\ol{A}$ and $p_{sM}=\varepsilon_{BA}\ot\id_{sM}\ot\varepsilon_{BA}:BA\ot sM\ot BA\twoheadrightarrow sM$ are natural projections.
Then $d$ is a differential on $T(s\ol{A}\oplus sM)$, i.e., a coderivation on $T(s\ol{A}\oplus sM)$ of degree $-1$ such that $d^2=0$,
if and only if $\az$ is a Hochschild 2-cocycle of $A$ with coefficients in $M$.

Now assume that $\az$ is a Hochschild 2-cocycle of $A$ with coefficients in $M$.
Then $d$ is a differential on $T(s\ol{A}\oplus sM)$.
Thus $T(s\ol{A}\oplus sM)$ is a coaugmented dg $K$-coalgebra, and $A\oplus M$ is an augmented $A_\infty$-algebra
with unit $\eta: K\xrightarrow{\eta_A}A\hookrightarrow A\oplus M$
and augmentation $\varepsilon: A\oplus M\twoheadrightarrow A\xrightarrow{\varepsilon_A}K$.
Here, $A\hookrightarrow A\oplus M$ and $A\oplus M\twoheadrightarrow A$ are the natural inclusion and projection.
We denote this $A_\infty$-algebra $A\oplus M$ by $\T(A,M,\alpha)$,
and call it the {\it Hochschild extension} of $A$ by $M$ and $\az$.
Obviously, the natural projection $\T(A,M,\alpha)\twoheadrightarrow A$ is a strict augmented $A_\infty$-algebra morphism.

Next, we show that two equivalent Hochschild 2-cocycles define isomorphic Hochschild extensions.
Let $\alpha, \alpha'\in C^\bullet(A,M)_{-2}$ be two equivalent Hochschild 2-cocycles,
i.e., there exists $\bz\in C^\bullet(A,M)_{-1}$ such that $d(\bz)=\alpha-\alpha^\prime$.
Then $\bz$ defines an augmented $A_\infty$-algebra isomorphism from $\T(A,M,\alpha)$ to $\T(A,M,\alpha^\prime)$,
or equivalently, a coaugmented dg $K$-coalgebra isomorphism
$$f_\bz : B\T(A,M,\alpha)\to B\T(A,M,\alpha^\prime).$$
Indeed, let $f_\bz: B\T(A,M,\alpha)\to B\T(A,M,\alpha^\prime)$ be the coaugmented dg $K$-coalgebra morphism uniquely determined by the dg $K$-bimodule morphism
$$B\T(A,M,\alpha)=T(s\ol{A}\oplus sM)\twoheadrightarrow BA\oplus BA\ot sM\ot BA
\xrightarrow{\begin{bmatrix}p_{s\ol{A}} & 0\\ s\circ\bz &p_{sM}\end{bmatrix}} s\ol{A}\oplus sM.$$
and $g_\bz : B\T(A,M,\alpha^\prime) \to B\T(A,M,\alpha)$ the coaugmented dg $K$-coalgebra morphism
uniquely determined by the dg $K$-bimodule morphism
$$B\T(A,M, \alpha')=T(s\ol{A}\oplus sM)\twoheadrightarrow BA\oplus BA\ot sM\ot BA
\xrightarrow{\begin{bmatrix}p_{s\ol{A}} & 0\\ -s\circ\bz & p_{sM}\end{bmatrix}}s\ol{A} \oplus sM.$$
The condition $d(\bz)=\alpha-\alpha^\prime$ ensures that both $f_\bz$ and $g_\bz$ are coaugmented dg $K$-coalgebra morphisms.
It is not difficult to check that $f_\bz$ and $g_\bz$ are inverse to each other.
So far we have proved the following theorem:

\begin{theorem} \label{Theorem-Hoch-Ext}
Let $A$ be an augmented dg $K$-algebra, $M$ a dg $A$-bimodule with degree $-2$ component zero, i.e., $M_{-2}=0$,
and $\az,\az'\in C^\bullet(A,M)_{-2}$ two Hochschild cochains of degree $-2$.
Then $\T(A,M,\az)$ is an augmented $A_\infty$-algebra if and only if $\az$ is a Hochschild 2-cocycle.
Moreover, $\T(A,M,\az) \cong \T(A,M,\az')$ if $[\az]=[\az']\in H^2(A,M)$.
\end{theorem}

\begin{remark}{\rm
In Theorem \ref{Theorem-Hoch-Ext}, the assumption $M_{-2} =0$ is necessary.
Otherwise, it is possible that $d_{B\T(A,M,\az)}(1_{B\T(A,M,\az)})
= -s\az(1_{BA})\in(sM)_{-1}=M_{-2}$ is nonzero, which leads to a curved $A_\infty$-algebra $\T(A,M,\az)$.
}\end{remark}

\bigskip

\noindent{\bf Exact Hochschild extensions.} A special kind of Hochschild extensions of augmented dg $K$-algebras,
called {\it symmetric Hochschild extensions} \cite{OhnTakYam99}, are symmetric $A_\infty$-algebras.

\begin{definition}{\rm
An augmented $A_\infty$-algebra $A$ is {\it $n$-symmetric} if $A \cong A^\vee[-n]$ as $A_\infty$-$A$-bimodules.
}\end{definition}

\begin{remark}{\rm
There are other three closely related concepts: $\infty$-Poincar\'{e} duality structure \cite{Tra08},
$A_\infty$-cyclic structure \cite{Van15} and pre-Calabi-Yau algebra \cite{TraZei07,Sei17,KonVla13}.
}\end{remark}

The following result implies that {\it exact Hochschild extensions},
namely, the Hochschild extensions defined by Hochschild 2-cocycles in exact Hochschild cohomology classes,
are always symmetric Hochschild extensions.

\begin{theorem} \label{Theorem-HochExt-Sym}
Let $A$ be a finite dimensional complete typical dg $K$-algebra, $n\in \mathbb{Z}$ satisfying $A_{2-n}=0$, and $[\az]\in H^2(A,A^\vee[-n])$ exact.
Then $\T _n(A,\az):=\T (A,A^\vee[-n],\az)$ is an $n$-symmetric $A_\infty$-algebra.
\end{theorem}

\begin{proof}
For simplicity, we denote the exact Hochschild extension $\T _n(A,\alpha)$ by $\T $, which is an augmented $A_\infty$-algebra by Theorem \ref{Theorem-Hoch-Ext}.
We need to define an $A_\infty$-$\T$-bimodule isomorphism between $\T$ and $\T^\vee[-n]$,
or equivalently, a dg $B\T$-bicomodule isomorphism between $B\T\ot\T\ot B\T$ and $B\T\ot\T ^\vee[-n]\ot B\T$.
By Lemma \ref{Lemma-Free-Bicom-Hom}, we have dg $k$-module isomorphisms
$$\begin{array}{ll}
& \Hom_{(B\T)^e}(B\T\ot\T\ot B\T, B\T\ot\T^\vee[-n]\ot B\T) \vspace{2mm}\\
\cong & s^{-n}C_\bullet(\T,\T\ot B\T\ot\T)^\vee \vspace{2mm}\\
= & s^{-n}((\T\ot B\T\ot\T)\ot_{K^e} B\T)^\vee \vspace{2mm}\\
\cong & s^{-n}(\T\ot_{K^e} (B\T\ot\T\ot B\T))^\vee.
\end{array}$$
We will define a 0-cycle in $s^{-n}(\T\ot_{K^e} (B\T\ot\T\ot B\T))^\vee$
which corresponds to a dg $B\T$-bicomodule isomorphism from $B\T\ot\T\ot B\T$ to $B\T\ot\T^\vee[-n]\ot B\T$.

According to the graded $k$-module isomorphism
$$\Hom_{(BA)^e}(BA\ot A^\vee\ot BA, BA\ot A^\vee\ot BA)\cong (A\ot_{K^e}( BA\ot A^\vee\ot BA))^\vee,$$
we can endow the graded $k$-module $(A\ot_{K^e}(BA\ot A^\vee\ot BA))^\vee$
with the differential induced from that of the dg $k$-module $\Hom_{(BA)^e}(BA\ot A^\vee\ot BA, BA\ot A^\vee\ot BA)$.
The identity morphism on $BA\ot A^\vee\ot BA$ corresponds to the 0-cycle
$$\begin{array}{ll}
\theta& = \sum\limits_{e_i\in I}(e_i\ot 1\ot e_i^\vee\ot 1)^\vee + \sum\limits_{a\in S}(-1)^{|a|}(a\ot 1\ot a^\vee\ot 1)^\vee \vspace{2mm}\\
& \in (A\ot_{K^e}(BA\ot A^\vee\ot BA))^\vee
\end{array}$$
where $I=\{e_1,\cdots,e_t\}$ is a complete set of primitive idempotents of $K=k^t$,
$S=\bigcup\limits_{1\le i,j\le t} S_{ij}$, $S_{ij}$ is a $k$-basis of $e_i\ol{A}e_j$ for all $1\le i,j\le t$,
and $I^\vee=\{e_1^\vee,\cdots,e_t^\vee\}$ and $S^\vee=\{a^\vee\ |\ a\in S\}$ are the dual bases of $I$ and $S$ respectively.
Then we have an $(n-1)$-cycle
$$\begin{array}{ll}
s^{n-1}\theta & =\sum\limits_{e_i\in I}(e_i\ot 1\ot s^{1-n}e_i^\vee\ot 1)^\vee + \sum\limits_{a\in S}(-1)^{n|a|}(a\ot 1\ot s^{1-n}a^\vee\ot 1)^\vee
\vspace{2mm}\\ & \in (A\ot_{K^e}(BA\ot s^{1-n}A^\vee\ot BA))^\vee.
\end{array}$$
The canonical projections $\T=A\oplus s^{-n}A^\vee \twoheadrightarrow A$ and $B\T = T(s\ol{A}\oplus s^{1-n}A^\vee) \twoheadrightarrow BA\ot s^{1-n}A^\vee\ot BA$
define a projection $p: \T \ot_{K^e} B\T \twoheadrightarrow A\ot_{K^e} (BA\ot s^{1-n}A^\vee\ot BA)$, which is a surjective graded $k$-module morphism.
Its dual
$$p^\vee:(A\ot_{K^e}(BA\ot s^{1-n}A^\vee\ot BA))^\vee\rightarrow (\T\ot_{K^e} B\T)^\vee,$$
is an injective graded $k$-module morphism. Thus
$$\begin{array}{ll}
p^\vee(s^{n-1}\theta) & =\sum\limits_{e_i\in I}(e_i\ot s^{1-n}e_i^\vee)^\vee + \sum\limits_{a\in S}(-1)^{n|a|}(a\ot s^{1-n}a^\vee)^\vee \vspace{2mm}\\
& \in (\T\ot_{K^e} B\T)^\vee
\end{array}$$
is of degree $n-1$. Let
$$\tilde{\theta}:=(-1)^{n-1}p^\vee(s^{n-1}\theta) \quad \in (\T\ot_{K^e} B\T)^\vee$$
which is of degree $n-1$ as well.

The dual
$\wt{B}_\T^\vee: (\T\ot_{K^e} B\T)^\vee\to (\T\ot_{K^e}(B\T\ot\T\ot B\T))^\vee$
of the lift $\wt{B}_\T : \T\ot_{K^e}(B\T\ot\T\ot B\T) \to \T\ot_{K^e} B\T$ of Connes operator $B_\T : \T\ot_{K^e} B\T \to \T\ot_{K^e} B\T$ maps $\tilde{\theta}$ to
$$\begin{array}{ll}
\wt{B}^\vee_\T(\tilde{\theta})& = \sum\limits_{e_i\in I}((e_i\ot 1\ot s^{-n}e_i^\vee\ot1)^\vee+(s^{-n}e_i^\vee\ot 1\ot e_i\ot 1)^\vee) \vspace{2mm}\\
& \quad \quad +\sum\limits_{a\in S}((-1)^{n|a|+|a|}( a\ot 1\ot s^{-n}a^\vee\ot 1)^\vee+(s^{-n}a^\vee\ot1\ot a\ot 1)^\vee) \vspace{2mm}\\
& \in (\T \ot_{K^e}(B\T \ot \T \ot B\T))^\vee \end{array}$$
which is of degree $n$.

The natural projection $\T=A\oplus s^{-n}A^\vee \twoheadrightarrow A$ is a strict augmented $A_\infty$-algebra morphism.
It induces a surjective $k$-module complex morphism
$$\pi: \T \ot_{K^e} B\T \rightarrow A\ot_{K^e} BA$$
between the Hochschild chain complexes of $\T$ and $A$, and further an injective $k$-module complex morphism
$$\pi^\vee:(A\ot_{K^e} BA)^\vee\rightarrow (\T \ot_{K^e} B\T )^\vee.$$
The Hochschild 2-cocycle $\az\in Z_{-2}\Hom_{K^e}(BA,A^\vee[-n])$ induces an $(n-2)$-cycle $s^n\az\in Z_{n-2}\Hom_{K^e}(BA,A^\vee)=Z_{n-2}(A\ot_{K^e}BA)^\vee$.
It is easy to check that
$$d_{(\T\ot_{K^e} B\T)^\vee}(\tilde{\theta})=-b^\vee_\T(\tilde{\theta})=-b^\vee_\T((-1)^{n-1}p^\vee(s^{n-1}\theta))=\pi^\vee(s^n\az)$$
where $b_\T$ is the differential of the Hochschild chain complex $\T\ot_{K^e} B\T$ of $\T$.

Since $[\az]\in H^2(A,A^\vee[-n])=HH_{2-n}(A)^\vee=H_{2-n}(A\ot_{K^e}BA)^\vee$ is exact,
$[s^n\alpha]\in H_{2-n}(A\ot_{K^e}BA)^\vee$ admits a lift $[\eta]\in HC_{2-n}(A)^\vee=H_{n-2}(CC_{\bullet}(A)^\vee)=H_{n-2}(C_{\bullet}(A)^\vee[[u]])$
along the map $I_{2-n}^\vee : HC_{2-n}(A)^\vee \to HH_{2-n}(A)^\vee$.
Suppose
$$\eta=\eta^0+\eta^1u+\eta^2u^2+\cdots$$
in $CC_{\bullet}(A)^\vee=(C_{\bullet}(A)^\vee[[u]],-b^\vee_A-B^\vee_A u)$,
where $\eta^i \in C_{\bullet}(A)^\vee=(A\ot_{K^e}BA)^\vee$, $|\eta^i|=n-2+2i$, $\eta^0=s^n\alpha$, and $b^\vee_A(\eta^{i+1})+B^\vee_A(\eta^i)=0$ for all $i\ge 0$.

Now we want to lift $\eta^1\in (A\ot_{K^e}BA)^\vee$ to $\tilde{\eta}^1 \in (A\ot_{K^e}(BA\ot A\ot BA))^\vee$
along the map $\delta^\vee_A: (A\ot_{K^e}(BA\ot A\ot BA))^\vee \to (A\ot_{K^e}BA)^\vee$,
i.e., find an element $\tilde{\eta}^1\in (A\ot_{K^e}(BA\ot A\ot BA))^\vee$ such that
$\delta^\vee_A(\tilde{\eta}^1)=\eta^1$, and
$\wt{B}^\vee_\T(\tilde{\theta})+\tilde{q}^\vee(\tilde{\eta}^1) \in (\T\ot_{K^e}(B\T\ot\T\ot B\T))^\vee$ is an $n$-cycle,
where $\tilde{q}^\vee:(A\ot_{K^e}(BA\ot A\ot BA))^\vee\rightarrow (\T \ot_{K^e}(B\T\ot\T\ot B\T))^\vee$
is the dual of the $k$-module complex morphism $\tilde{q}:\T \ot_{K^e}(B\T\ot\T\ot B\T) \to A\ot_{K^e}(BA\ot A\ot BA)$
induced by the natural projection $q:\T \twoheadrightarrow A$.

Keep in mind the following commutative diagram:
$$\xymatrix{
(A\ot_{K^e}BA)^\vee \ar[rr]^-{\wt{B}_A^\vee} \ar[dr]^-{B_A^\vee} \ar[dd]^-{\pi^\vee} && (A\ot_{K^e}(BA\ot A\ot BA))^\vee \ar[dl]^-{\dz_A^\vee} \ar[dd]^-{\wt{q}^\vee} \\
& (A\ot_{K^e}BA)^\vee \ar[dd]^(.3){\pi^\vee} & \\
(\T\ot_{K^e}B\T)^\vee \ar[rr]^(.3){\wt{B}_{\T}^\vee} \ar[dr]^-{B_{\T}^\vee} && (\T\ot_{K^e}(B\T\ot\T\ot B\T))^\vee \ar[dl]^-{\dz_{\T}^\vee}\\
& (\T\ot_{K^e}B\T)^\vee &
}$$

From $B^\vee_A(\eta^0)+b^\vee_A(\eta^1)=0$, we know $B^\vee_A(\eta^0)\in (A\ot_{K^e}BA)^\vee$ is a boundary.
Since $B^\vee_A(\eta^0)=\delta^\vee_A\wt{B}^\vee_A(\eta^0)$ and
$\delta^\vee_A:(A\ot_{K^e}(BA\ot A\ot BA))^\vee \twoheadrightarrow (A\ot_{K^e}BA)^\vee$ is a surjective quasi-isomorphism,
$\wt{B}^\vee_A(\eta^0) \in (A\ot_{K^e}(BA\ot A\ot BA))^\vee$ is also a boundary.
Thus there exists $\tilde{\eta}' \in (A\ot_{K^e}(BA\ot A\ot BA))^\vee$ such that $\wt{B}^\vee_A(\eta^0)=-\tilde{b}^\vee_A(\tilde{\eta}')$
where $\tilde{b}_A$ is the differential of $A\ot_{K^e}(BA\ot A\ot BA)$.
Since $\delta^\vee_A\wt{B}_A(\eta^0)=-\delta^\vee_A\tilde{b}^\vee_A(\tilde{\eta}')$,
we have $B^\vee_A(\eta^0)=-b^\vee_A\delta^\vee_A(\tilde{\eta}')$.
On the other hand, $B^\vee_A(\eta^0)+b^\vee_A(\eta^1)=0$, so we have $b^\vee_A\delta^\vee_A(\tilde{\eta}')=b^\vee_A(\eta^1)$,
which implies that $\delta^\vee_A(\tilde{\eta}')-\eta^1 \in (A\ot_{K^e}BA)^\vee$ is an $n$-cycle.
Since $\delta^\vee_A$ is a surjective quasi-isomorphism, it induces a surjection on cycles:
$$\delta^\vee_A:Z_\bullet(A\ot_{K^e}(BA\ot A\ot BA))^\vee \twoheadrightarrow Z_{\bullet}(A\ot_{K^e}BA)^\vee.$$
Thus there exists $\tilde{\eta}''\in Z_n(A\ot_{K^e}(BA\ot A\ot BA))^\vee$, which implies $\tilde{b}^\vee_A(\tilde{\eta}'')=0$, such that
$\delta^\vee_A(-\tilde{\eta}'')=\delta^\vee_A(\tilde{\eta}')-\eta^1$, i.e., $\delta^\vee_A(\tilde{\eta}'+\tilde{\eta}'')=\eta^1$. Let $\tilde{\eta}^1=\tilde{\eta}'+\tilde{\eta}''$.
Then $\delta^\vee_A(\tilde{\eta}^1)=\eta^1$.

Since
$b^\vee_\T(\wt{B}^\vee_\T(\tilde{\theta})+\tilde{q}^\vee(\tilde{\eta}^1))
= -\wt{B}^\vee_\T b^\vee_\T(\tilde{\theta})+\tilde{b}^\vee_\T\tilde{q}^\vee(\tilde{\eta}^1)
=\wt{B}^\vee_\T(\pi^\vee(\eta^0))+\tilde{q}^\vee(\tilde{b}^\vee_A\tilde{\eta}^1)
=\tilde{q}^\vee(\wt{B}^\vee_A(\eta^0)+\tilde{b}^\vee_A(\tilde{\eta}^1))$
and
$\wt{B}^\vee_A(\eta^0)+\tilde{b}^\vee_A(\tilde{\eta}^1)
=\wt{B}^\vee_A(\eta^0)+\tilde{b}^\vee_A(\tilde{\eta}'+\tilde{\eta}'')
=(\wt{B}^\vee_A(\eta^0)+\tilde{b}^\vee_A(\tilde{\eta}'))+\tilde{b}^\vee_A(\tilde{\eta}'')
=0,$
we have $b^\vee_\T(\wt{B}^\vee_\T(\tilde{\theta})+\tilde{q}^\vee(\tilde{\eta}^1))=0$,
i.e., $\wt{B}^\vee_\T(\tilde{\theta})+\tilde{q}^\vee(\tilde{\eta}^1)\in \T \ot_{K^e}(B\T\ot\T\ot B\T)$ is an $n$-cycle.

The element $s^{-n}(\wt{B}_{\T}^\vee(\tilde{\theta})+\tilde{q}^\vee(\tilde{\eta}^1)) \in s^{-n}(\T\ot_{K^e}(B\T\ot\T\ot B\T))^\vee$ is a 0-cycle.
It corresponds to a dg $B\T$-bicomodule isomorphism $\phi:B\T\ot\T\ot B\T \rightarrow B\T\ot s^{-n}\T^\vee\ot B\T$.
Indeed, by Lemma \ref{Lemma-A-bimod-iso}, it is enough to show that $\phi_{0,0}:K\ot\T\ot K\to K\ot s^{-n}\T^\vee\ot K$ is a dg $K$-bimodule isomorphism.
Note that $s^{-n}\mathbb{T}^\vee= s^{-n}(A\oplus s^{-n}A^\vee)^\vee\cong s^{-n}A^\vee\oplus A=\T$,
where we identify $(s^{-n}A^\vee)^\vee$ with $s^nA$ by defining  $s^na(s^{-n}f)=(-1)^{|a|(n+|f|)}f(a)$ for all $a\in A$.
Now we consider the restricted map $\phi_{0,0}: \mathbb{T}\rightarrow s^{-n}\mathbb{T}^\vee$ of the map $\phi$ induced by
$s^{-n}\wt{B}_{\T}^\vee(\tilde{\theta})+s^{-n}\tilde{q}^\vee(\tilde{\eta}^1)$.
The effect of $s^{-n}\wt{B}^\vee(\tilde{\theta})$ is identifying $\mathbb{T}$ with $s^{-n}\mathbb{T}^\vee$.
The effect of $s^{-n}\tilde{q}^\vee(\tilde{\eta}^1)$ is sending $A$ in $\mathbb{T}$ to $s^{-n}A^\vee$ in $s^{-n}\mathbb{T}^\vee$,
and sending $s^{-n}A^\vee$ in $\mathbb{T}$ to zero.
Thus $\phi_{0,0}$ is a dg $K$-bimodule isomorphism.
\end{proof}

The following cohomological criterion of symmetric Hochschild extension generalizes \cite[Theorem 1]{OhnTakYam99}.

\begin{proposition} \label{Proposition-OhnTakYam-Gen}
Let $A$ be a finite dimensional elementary $k$-algebra and $[\az]\in H^2(A,A^\vee)$ exact.
Then $\T (A,\az)$ is a symmetric algebra.
\end{proposition}

\begin{proof} Since $[\alpha] \in H^2(A,A^\vee)$ is exact, it follows from Theorem \ref{Theorem-HochExt-Sym} that
$\T=\T(A,\az)$ is a 0-symmetric $A_\infty$-algebra.
Thus there is a dg $B\T$-bicomodule isomorphism $\phi:B\T\ot\T\ot B\T \to B\T\ot \T^\vee\ot B\T$.
Since $\T$ is concentrated in degree 0, we have $\Hom_{K^e}(B\T\ot\T\ot B\T,\T^\vee)_0\cong \Hom_{K^e}(\T,\T^\vee)$.
Thus $\phi$ must be given by a map $\phi_{0,0}: \T\to\T^\vee$.
The compatibility of $\phi$ with differential implies that $\phi_{0,0}$ is a $\T$-bimodule morphism.
Since $\phi$ is a bijection, $\phi_{0,0}$ is a bijection too.
Thus $\phi$ is a $\T$-bimodule isomorphism. So $\T$ is a symmetric algebra.
\end{proof}

\section{Koszul duality}

In this section, utilizing the relation between the Hochschild homologies of a dg $K$-algebra and its Koszul dual,
we introduce the Koszul dual of a Hochschild cohomology class.
Employing the relations between the Hochschild (co)homologies and (negative) cyclic homologies of a dg $K$-algebra and its Koszul dual,
we set up the correspondence between the exact Hochschild cohomology classes of a dg $K$-algebra
and the almost exact Hochschild homology classes of its Koszul dual,
which is crucial for studying the Koszul duals of exact Hochschild extensions.

\subsection{Hochschild (co)homology and Koszul duality}

The relations between the Hochschild (co)homologies of a dg $K$-algebra and its Koszul dual were clarified in \cite{HanLiuWan18}.

\begin{theorem} \label{Theorem-HHHH_Connes} {\rm (\cite[Theorem 8]{HanLiuWan18})}
Let $A$ be a complete typical dg $K$-algebra. Then

{\rm (1)} there is an isomorphism $h^\bullet : HH^\bullet(A) \rightarrow HH^\bullet(A^\dag)$ of Gerstenharber algebras;

{\rm (2)} there is an isomorphism
$h_\bullet : HH_{\bullet}(A)^\vee \rightarrow HH_{-\bullet}(A^\dag)$ of graded $k$-modules
such that the following diagram is commutative:
$$\xymatrix{ HH_{\bullet}(A)^\vee \ar[r]^-{B^\vee} \ar[d]^-{h_\bullet} & HH_{\bullet}(A)^\vee \ar[d]^-{h_\bullet}\\
HH_{-\bullet}(A^\dag) \ar[r]^-{B} & HH_{-\bullet}(A^\dag). }$$
\end{theorem}

\bigskip

Thanks to Theorem \ref{Theorem-HHHH_Connes}, we have the following concept which is crucial for studying the Koszul duals of exact Hochschild extensions.

\begin{definition}{\rm
Let $A$ be a complete typical dg $K$-algebra and $[\az] \in H^\bullet(A,A^\vee)=HH_\bullet(A)^\vee$.
The {\it Koszul dual} of the Hochschild cohomology class $[\az]$ is the Hochschild homology class $[\az^\dag] := h_\bullet([\az])\in HH_\bullet(A^\dag)$.
}\end{definition}

\subsection{Cyclic homology and Koszul duality}

Now we clarify the relation between the cyclic homologies of a complete typical dg $K$-algebra $A$ and its Koszul dual $A^\dag$.

Let ${C_{\bullet}(A)}^\vee$ be the graded dual of the Hochschild complex $C_{\bullet}(A)$ of $A$.
Then ${C_{\bullet}(A)}^\vee$ is a dg $\Lambda$-module with the action of $\epsilon$ given by the operator $B^\vee$.
By the proof of \cite[Theorem 8]{HanLiuWan18}, we have the following commutative diagram
$$\xymatrix{
C_{-\bullet}(A^\dag)\ar[d]_-{\varrho}\ar[r]^-{B_{A^\dag}}&C_{-\bullet}(A^\dag)\ar[d]^-{\varrho}\\
C_{\bullet}(A)^\vee \ar[r]^-{B^\vee_A} & C_{\bullet}(A)^\vee
}$$
where $\varrho:=\omega_{A,BA} \circ (\psi \otimes \id) \circ \tilde{\tilde{\rho}}$.
It implies that the map $\varrho$ is a quasi-isomorphism of dg $\Lambda$-modules.

Acting the functors $C_\bullet(A^\dag)[[u]] \hat{\ot}_{k[[u]]} -$ and $C_\bullet(A)^\vee[[u]] \hat{\ot}_{k[[u]]} -$ on the
the short exact sequence
$$0 \to uk[[u]] \to k[[u]] \to k[[u]]/uk[[u]] \to 0$$
of $k[[u]]$-modules, we get the following commutative diagram
$$\xymatrix{
0 \ar[r] & CN_{-\bullet}(A^\dag)[-2] \ar[r] \ar[d] &  CN_{-\bullet}(A^\dag) \ar[r] \ar[d] & C_{-\bullet}(A^\dag) \ar[r] \ar[d] & 0 \\
0 \ar[r] & CC_{\bullet}(A)^\vee[-2] \ar[r] & CC_{\bullet}(A)^\vee \ar[r] & C_{\bullet}(A)^\vee \ar[r] & 0.
}$$
with exact rows and quasi-isomorphic columns.
Indeed, the right column is just the quasi-isomorphism $\varrho : C_{-\bullet}(A^\dag) \to C_{\bullet}(A)^\vee$,
the middle column is the quasi-isomorphism $CN_{-\bullet}(A^\dag)=C_{-\bullet}(A^\dag)[[u]] \xrightarrow{\varrho[[u]]} 
C_{\bullet}(A)^\vee[[u]] \cong CC_{\bullet}(A)^\vee$ induced by $\varrho$,
and the left column is the shift of the middle one.

By taking homology, we obtain the following result:

\begin{proposition} \label{Proposition-HC-KoszulDual}
Let $A$ be a complete typical dg $K$-algebra and $A^\dag$ its Koszul dual. Then there exists a graded $k$-vector space isomorphism $c_\bullet : HC_\bullet(A)^\vee \to HN_{-\bullet}(A^\dag)$ such that the following diagram is commutative:
$$\xymatrixcolsep{1pc}\xymatrixrowsep{2pc}\xymatrix{
\cdots \ar[r] & HH_{n-1}(A)^\vee \ar[r]^-{B^{\vee}_n} \ar[d]^-{h_{n-1}}_-{\cong} & HC_{n-2}(A)^\vee \ar[r]^-{S^\vee_n} \ar[d]^-{c_{n-2}}_-{\cong} & HC_n(A)^\vee \ar[r]^-{I^\vee_n} \ar[d]^-{c_n}_-{\cong} & HH_n(A)^\vee \ar[r] \ar[d]^-{h_n}_-{\cong} & \cdots \\
\cdots \ar[r] & HH_{-n+1}(A^\dag) \ar[r]^-{B^{'\vee}_{-n}} & HN_{-n+2}(A^\dag) \ar[r]^-{S'_{-n}} & HN_{-n}(A^\dag) \ar[r]^-{P_{-n}} & HH_{-n}(A^\dag) \ar[r] & \cdots
}$$
\end{proposition}

\bigskip

Applying Proposition \ref{Proposition-HC-KoszulDual}, we can obtain the correspondence
between the exact Hochschild cohomology classes of $A$ with coefficients in $A^\vee$ and the almost exact Hochschild classes of $A^\dag$.

\begin{proposition} \label{Proposition-Ex-AlmostEx}
Let $A$ be a complete typical dg $K$-algebra. Then the Hochschild cohomology class $[\az]\in H^n(A,A^\vee) = HH_n(A)^\vee$ is exact if and only if
its Koszul dual Hochschild homology class $[\az^\dag]\in HH_{-n}(A^\dag)$ is almost exact.
\end{proposition}

\begin{proof}
By Proposition \ref{Proposition-HC-KoszulDual}, we have the following commutative diagram:
$$\xymatrix{ HC_n(A)^\vee \ar[r]^-{I_n^\vee} \ar[d]_-{\cong}^-{c_n} & HH_n(A)^\vee \ar[d]_-{\cong}^-{h_n} \ar@{=}[r] & H^n(A,A^\vee) \\
HN_{-n}(A^\dag) \ar[r]^-{P_{-n}} & HH_{-n}(A^\dag) & }$$
Thus $[\az]$ is exact, i.e.,
$[\az]\in\Im I_n^\vee$, if and only if $[\az^\dag]=h_n([\az]) \in\Im (h_n\circ I_n^\vee) = \Im (P_{-n} \circ c_n) = \Im P_{-n}$,
if and only if $[\az^\dag]$ is almost exact.
\end{proof}

\section{Deformed Calabi-Yau completions}

In this section, we show that the Koszul dual of trivial extension is Calabi-Yau completion,
and the Koszul dual of exact Hochschild extension is deformed Calabi-Yau completion.

\subsection{Calabi-Yau dg algebras}

Let $A$ be a homologically smooth dg $K$-algebra and $A^e=A^\op\ot_k A$ its enveloping dg algebra.
Then the derived Hom-functor $\RHom_{A^e}(-,A^e)$ induces a dual on the perfect derived category $\per (A^e)$ of $A^e$.
Thus we have isomorphisms
$\RHom_{A^e}(\RHom_{A^e}(A,A^e)[n],A) \cong A\ot^L_{A^e}\RHom_{A^e}(\RHom_{A^e}(A,A^e),A^e)[-n] \cong A\ot^L_{A^e}A[-n]$.
Therefore, the morphisms in $\Hom_{\mathcal{D}(A^e)}(\RHom_{A^e}(A,A^e)[n],A)$ correspond bijectively to the Hochschild homology classes in $HH_n(A)$,
where $\mathcal{D}(A^e)$ is the unbounded derived category of $A^e$ (see \cite{Kel94}).

\begin{definition}{\rm (\cite{Gin06,Van15,Her18,dTdVVdB18})
A homologically smooth dg $K$-algebra $A$ is {\it Calabi-Yau of dimension $n$} or {\it $n$-Calabi-Yau}
if there is an isomorphism
$$\xi : \RHom_{A^e}(A,A^e)[n] \to A$$
in the derived category $\mathcal{D}(A^e)$ of $A$-bimodules.
An $n$-Calabi-Yau dg algebra $A$ is {\it almost exact} if $[\xi] \in HH_n(A)$ is almost exact,
and {\it exact} if $[\xi] \in HH_n(A)$ is exact.
}\end{definition}

Obviously, exact Calabi-Yau dg algebras are almost exact. It is a common feeling that
almost exact Calabi-Yau dg algebra should be the ``correct'' definition of Calabi-Yau dg algebra (see \cite[Page 1264]{dTdVVdB18}).

\bigskip

\begin{definition}{\rm (\cite{Kel11})
Let $A$ be a homologically smooth dg $K$-algebra.
The {\it $n$-Calabi-Yau completion} or {\it derived $n$-preprojective dg algebra} $\Pi_n(A)$ of $A$ is the tensor dg $K$-algebra
$T_A(\theta) = \bigoplus\limits_{i=0}^\infty\theta^{\ot_Ai} = A\oplus \theta \oplus (\theta\ot_A\theta) \oplus \cdots$,
where $\theta=s^{n-1}\Theta$ and the {\it inverse dualizing complex} $\Theta$ is the cofibrant resolution of the dg $A$-bimodule $\RHom_{A^e}(A,A^e)$.

For a Hochschild $(n-2)$-cycle $\az$ of $A$, the {\it deformed Calabi-Yau completion} $\Pi_n(A,\az)$ of $A$ by $\az$
is the tensor dg $K$-algebra $T_A(\theta)$ with differential $\tilde{d}=d+d_{\az}$,
where $d$ is the differential of the tensor dg $K$-algebra $T_A(\theta)$ and $d_{\az}$ is induced by $\alpha$ as follows: Since
$A\ot_{A^e}^LA[2-n]\cong \RHom_{A^e}(\RHom_{A^e}(A,A^e)[n-2],A)\cong \Hom_{A^e}(\theta[-1],A)$,
$\az$ determines a closed $A$-bimodule morphism $c_\alpha:\theta\to A$ of degree $-1$.
The derivation $d_{\az}$ of $T_A(\theta)$ is induced by the composition $\theta\xrightarrow{c_\alpha}A\hookrightarrow T_A(\theta)$.
}\end{definition}

\begin{theorem}\label{Theorem-Keller} {\rm (\cite[Theorem 4.8]{Kel11} and \cite[Theorem 1.1]{Kel18})}
Let $A$ be a homologically smooth dg $K$-algebra. Then the Calabi-Yau completion $\Pi_n(A)$ of $A$ is an exact $n$-Calabi-Yau dg algebra.
\end{theorem}

\begin{theorem}\label{Theorem-Yeung} {\rm (\cite[Theorem 3.17]{Yeu16})}
Let $A$ be a homologically smooth dg algebra and $[\az] \in HH_{n-2}(A)$ almost exact.
Then the deformed Calabi-Yau completion $\Pi_n(A,\az)$ of $A$ by $\az$ is an almost exact $n$-Calabi-Yau dg algebra.
\end{theorem}

\subsection{Trivial extensions and Calabi-Yau completions}

Now we show that the Koszul dual of trivial extension is Calabi-Yau completion.

Let $A$ be a finite dimensional complete dg $K$-algebra.
The {\it $n$-trivial extension} $\T_n(A)$ of $A$ is the augmented dg $K$-algebra $A\oplus A^\vee[-n]$
with the product given by $(a,f)\cdot(a',f'):=(aa',af'+fa')$ for all $a,a'\in A$ and $f,f'\in A^\vee[-n]$.

\begin{theorem}\label{Theorem-TrivExt-CYComp}
Let $A$ be a finite dimensional complete dg $K$-algebra. Then

{\rm (1)} the $n$-trivial extension $\T_n(A)$ of $A$ is an $n$-symmetric dg $K$-algebra,

{\rm (2)} the Koszul dual $A^\dag$ of $A$ is a homologically smooth dg $K$-algebra,

{\rm (3)} the $n$-Calabi-Yau completion $\Pi_n(A^\dag) \cong \T_n(A)^\dag$, and they are both exact $n$-Calabi-Yau dg algebras.
\end{theorem}

\begin{proof} (1) It is easy to see that $\T _n(A)=A \oplus A^\vee[-n] \cong (A \oplus A^\vee[-n])^\vee[-n] = \T_n(A)^\vee[-n]$
as dg $\T_n(A)$-bimodules, i.e., $\T_n(A)$ is strictly $n$-symmetric \cite{HanLiuWan18}.

(2) It is just \cite[Theorem 7]{HanLiuWan18}.

(3) It follows from \cite[Proposition 6]{HanLiuWan18} that the dg $A^\dag$-bimodule $A^\dag$
admits a semi-projective resolution $A^\dag\ot A^\vee \ot A^\dag$.
By \cite[Lemma 12]{HanLiuWan18}, we have isomorphisms
$\RHom_{A^{\dag e}}(A^\dag,A^{\dag e}) \cong \Hom_{A^{\dag e}}(A^\dag\ot A^\vee\ot A^\dag, A^{\dag e})
\cong A^\dag\ot A\ot A^\dag$ which is semi-projective.
Thus $\Theta=A^\dag\ot A\ot A^\dag$ and $\theta=s^{n-1}\Theta \cong A^\dag\ot s^{n-1}A\ot A^\dag$,
and further $\Pi_n(A^\dag)=T_{A^\dag}(\theta) = A^\dag\oplus(A^\dag\ot s^{n-1}A\ot A^\dag)\oplus
(A^\dag\ot s^{n-1}A \ot A^\dag \ot s^{n-1}A \ot A^\dag) \oplus \cdots$.

The $n$-trivial extension $\mathbb{T}_n(A)$ is the dg $K$-algebra $A\oplus s^{-n}A^\vee$
with product $(a,s^{-n}f)\cdot(b,s^{-n}g)=(ab,s^{-n}fb+(-1)^{n|a|}s^{-n}ag)$.
Its Koszul dual $\T_n(A)^\dag=\Omega(\T_n(A)^\vee)=T(s^{-1}\ol{\T_n(A)}^\vee)=T(s^{-1}(\ol{A}\oplus s^{-n}A^\vee)^\vee)
=T(s^{-1}(\ol{A}^\vee\oplus (s^{-n}A^\vee)^\vee)) \linebreak = T(s^{-1}\ol{A}^\vee\oplus s^{n-1}A)$.
So $\T_n(A)^\dag=K\oplus(s^{-1}\ol{A}^\vee\oplus s^{n-1}A)\oplus
(s^{-1}\ol{A}^\vee\oplus s^{n-1}A)^{\ot 2} \oplus \cdots$.

Due to $A^\dag=\Omega(A^\vee)=T(s^{-1}\ol{A}^\vee)=K\oplus s^{-1}\ol{A}^\vee \oplus (s^{-1}\ol{A}^\vee)^{\ot 2} \oplus \cdots$,
we can further decompose both $\Pi_n(A^\dag)$ and $\T_n(A)^\dag$
such that their direct summands equal correspondingly.
So we obtain a natural bijection $\Phi: \mathbb{T}_n(A)^\dag\rightarrow \Pi_n(A^\dag)$.
It is a graded $K$-algebra isomorphism and compatible with differentials.
Thus, $\Phi$ is a dg $K$-algebra isomorphism.

By Theorem \ref{Theorem-Keller}, i.e., \cite[Theorem 1.1]{Kel18}, or the proof of \cite[Theorem 3.30]{Yeu16},
we know both $\Pi_n(A^\dag)$ and $\T_n(A)^\dag$ are exact Calabi-Yau algebras.
\end{proof}

As an application of Theorem \ref{Theorem-TrivExt-CYComp}, we can recover \cite[Theorem 5.3]{Guo19}.

\begin{corollary}\label{Corollary-Guo} {\rm (\cite[Theorem 5.3]{Guo19})}
Let $A=kQ/I$ be a Koszul $n$-homogeneous bound quiver algebra, and the twisted trivial extension $A \ltimes (_\nu A)^\vee$ be quadratic.
Then the higher preprojective algebra $\Pi(A^!) \cong (A \ltimes (_\nu A)^\vee)^!$.
Here $\nu$ is the graded automorphism of $A$ sending $a\in Q_1$ to $(-1)^na$
and $_\nu A$ is the twisted $A$-bimodule given by $a\cdot b\cdot c:=\nu(a)bc$ for all $a,b,c\in A$.
\end{corollary}

\begin{proof}
First of all, for a coaugmented dg $K$-coalgebra $C$ with differential zero, we can define a new degree, called the {\it syzygy degree},
on the cobar construction $\Omega C$ of $C$ by $\omega(\langle c_1|\dots |c_n\rangle):=\sum\limits_{i=1}^n(|c_1|+1)$ (see \cite[3.3.2]{LodVal12}).
Then $\Omega C$ is still a dg $K$-algebra with respect to the syzygy degree but the differential is of degree $1$.

Write $A=K\oplus A_1\oplus\cdots\oplus A_n$. It is an augmented dg $K$-algebra with grading by length of path and differential zero.
Since $A$ is Koszul, analogous to \cite[Proposition 3.3.2]{LodVal12},
we have a dg $K$-algebra quasi-isomorphism $\phi:A^\dag\twoheadrightarrow A^!$
and $A^!\cong H^0(A^\dag)$ with respect to the syzygy degree on $A^\dag=\Omega(A^\vee)$.
Now $A^!$ is an ordinary algebra, i.e., it is concentrated on degree 0.

By Theorem \ref{Theorem-TrivExt-CYComp}, we have a dg $K$-algebra isomorphism $\Pi_{-n-1}(A^\dag)\cong \T_{-n-1}(A)^\dag$,
which is still a dg $K$-algebra isomorphism with respect to the syzygy degrees on $A^\dag$ and $\T_{-n-1}(A)^\dag$.
Taking the $0$-th cohomologies with respect to the syzygy degree on two sides of the isomorphism
$\Pi_{-n-1}(A^\dag)\cong\mathbb{T}_{-n-1}(A)^\dag$, we will obtain $\Pi(A^!) \cong (A \ltimes A^\vee_\nu)^!$.

Note that each $a_i\in A_i$ left acts on the dg $A$-bimodule
$s^{-n-2}A\subseteq s^{-1}\ol{A}^\vee\oplus s^{-1}(s^{n+1}A^\vee)^\vee=s^{-1}\ol{\T_{-n-1}(A)^\vee} \subseteq \T_{-n-1}(A)^\dag$
will create a sign $(-1)^{ni}$.
Taking the $0$-th cohomology of $\T_{-n-1}(A)^\dag$ with respect to the syzygy degree,
we obtain $H^0(\T_{-n-1}(A)^\dag)\cong(A \ltimes (_\nu A)^\vee)^!.$

Next, we compute the $0$-th cohomology of $\Pi_{-n-1}(A^\dag)$ with respect to syzygy degree.
By definition, $\Pi_{-n-1}(A^\dag)=T_{A^\dag}(A^\dag\otimes s^{-n-2}A\otimes A^\dag)
=A^\dag\oplus (A^\dag\otimes s^{-n-2}A\otimes A^\dag)\oplus (A^\dag\otimes s^{-n-2}A\otimes A^\dag)^{\otimes_{A^\dag} 2}\oplus \cdots$.
Note that $s^{-n-2}A$ is equal to $s^{-1}(s^{n+1}A^\vee)^\vee$ in $\T_{-n-1}(A)^\dag=\Omega(\T_{-n-1}(A)^\vee)$.
By the definition of syzygy degree on the cobar construction, we have $\omega(s^{-n-2}a_i) = |(s^{n+1}a_i^\vee)^\vee|+1=i-n$ for all $a_i\in A_i$.
In particular, $\omega(s^{-n-2}a_n)=0$.
The dg $K$-algebra quasi-isomorphism $\phi: A^\dag\rightarrow A^!$ induces a quasi-isomorphism
$\psi:A^\dag\otimes s^{-n-2}A\otimes A^\dag\rightarrow A^!\otimes s^{-n-2}A\otimes A^!$,
where $A^!\otimes s^{-n-2}A\otimes A^!$ is the cochain complex
\[0\rightarrow A^!\otimes K\otimes A^!\rightarrow A^!\otimes A_1\otimes A^!\rightarrow \dots \rightarrow A^!\otimes A_n\otimes A^!\rightarrow 0\]
with degree $i-n$ component $A^!\otimes A_i\otimes A^!$.
On the other hand, we consider the Koszul resolution $K(A^!)$ of the $A^!$-bimodule $A^!$
\[0\rightarrow A^!\otimes A_n^\vee\otimes A^!\otimes \rightarrow A^!\otimes A_{n-1}^\vee\otimes A^!\rightarrow \dots \rightarrow A^!\otimes A^! \rightarrow 0.\]
Then $A^!\otimes s^{-n-2}A\otimes A^!\cong\Hom_{(A^!)^e}(K(A^!),(A^!)^e)[n]$.
Taking the $0$-th cohomology with respect to syzygy degree, we get
$H^0(A^\dag\otimes s^{-n-2}A\otimes A^\dag) = H^0(A^!\otimes s^{-n-2}A\otimes A^!) = H^0(R\Hom_{(A^!)^e}(A^!,(A^!)^e)[n]) = \Ext^n_{(A^!)^e}(A^!,(A^!)^e).$
Denote $\Ext^n_{(A^!)^e}(A^!,(A^!)^e)$ by $E$ for short.
Since $A^\dag\otimes s^{-n-2}A\otimes A^\dag$ is non-positively graded with respect to the syzygy degree,
$H^0((A^\dag\otimes s^{-n-2}A\otimes A^\dag)^{\otimes_{A^\dag} m})= E^{\otimes_{A^!} m}$ for all $m\in \mathbb{N}$.
Therefore, $H^0(\Pi_{-n-1}(A^\dag))\cong T_{A^!}(E)=\Pi(A^!)$.
\end{proof}

\subsection{Exact Hochschild extensions and deformed Calabi-Yau completions}

Now we show that the Koszul dual of exact Hochschild extension is deformed Calabi-Yau completion.

\begin{theorem} \label{Theorem-ExHochExt-DefCYComp}
Let $A$ be a finite dimensional complete typical dg $K$-algebra, $n\in\mathbb{Z}$ satisfying $A_{2-n}=0$,
and $[\az] \in H^2(A,A^\vee[-n])$ exact. Then

{\rm (1)} the Hochschild extension $\T_n(A,\az)$ of $A$ by $\az$ is an $n$-symmetric $A_\infty$-algebra,

{\rm (2)} the Koszul dual $[\az^\dag]\in HH_{n-2}(A^\dag)$ of $\az$ is almost exact,

{\rm (3)} the deformed Calabi-Yau completion $\Pi_n(A^\dag,\az^\dag) \cong \T_n(A,\az)^\dag$, and they are both almost exact $n$-Calabi-Yau dg algebras.
\end{theorem}

\begin{proof} (1) This is just Theorem \ref{Theorem-HochExt-Sym}.

(2) This is just Proposition \ref{Proposition-Ex-AlmostEx}.

(3) For simplicity, we denote $\T_n(A,\az)$ by $\T$ and $\Pi_n(A^\dag,\az^\dag)$ by $\Pi$.

Since $A$ is finite dimensional typical and $|\alpha|=-2$, we have $\alpha(({s\ol{A}})^{\ot i})=0$ for $i\gg0$.
Thus the graded dual $\T^\vee$ of $\T$ is a coaugmented $A_\infty$-coalgebra.
Then we have $\T^\dag=\Omega(\T^\vee)=T(s^{-1}\ol{A}^\vee \oplus s^{-1}(s^{-n}A^\vee)^\vee) = T(s^{-1}\ol{A} \oplus s^{n-1}A)$.

The deformed Calabi-Yau completion of $A^\dag$ is $\Pi=T_{A^\dag}(A^\dag\ot s^{n-1}A \ot A^\dag)$ with differential $\tilde{d}=d+d_\az$,
where $d$ is the differential of the tensor dg $K$-algebra $T_{A^\dag}(A^\dag\ot s^{n-1}A \ot A^\dag)$
and $d_\az$ is the differential determined by the following map:
$$A^\dag \otimes s^{n-1}A \ot A^\dag \to A^\dag \hookrightarrow T_{A^\dag}(A^\dag \ot s^{n-1}A \ot A^\dag),\
1\ot s^{n-1}a\ot 1 \mt \az^\vee(s^na).$$

As in the proof of Theorem \ref{Theorem-TrivExt-CYComp}, we can decompose both $\T^\dag$ and $\Pi$ such that their direct summands equal correspondingly.
So we obtain a bijection $\Phi: \T^\dag \to \Pi$, which is a graded $K$-algebra morphism and compatible with differentials, i.e.,
a dg $K$-algebra isomorphism. By Theorem \ref{Theorem-Yeung}, i.e., \cite[Theorem 3.17]{Yeu16},
$\Pi_n(A^\dag,\az^\dag)$, and thus $\T_n(A,\az)^\dag$, is an almost exact $n$-Calabi-Yau dg algebras.
\end{proof}

\noindent {\footnotesize {\bf ACKNOWLEDGEMENT.} The authors are sponsored by Project 11571341 and 11971460 NSFC.}

\end{document}